\documentclass[final,leqno]{siamltex}

\usepackage{amsfonts}
\usepackage{amsmath}
\usepackage{graphicx}
\usepackage{color}
\usepackage{slashbox}
\usepackage[singlelinecheck=false % <-- important
]{caption}

\newtheorem{assumption}[theorem]{Assumption}

\title{An approximation scheme for
semilinear parabolic PDEs with convex and coercive Hamiltonians
\thanks{The authors thank the Editor, the Associate Editor, and the two referees for
their valuable comments and suggestions, and T. Lyons and C.
Reisinger for their helpful discussions. This work was presented at
seminars in Oxford University, Fudan University, Shanghai Jiao Tong
University and Jilin University. The authors thank the participants
for helpful comments and suggestions.}}

\author{Shuo Huang\thanks{Department of Statistics, The University of Warwick, Coventry CV4
7AL, U.K. \texttt{s.huang.13@warwick.ac.uk} } \and
Gechun Liang\thanks{%
Department of Statistics, The University of Warwick, Coventry CV4
7AL, U.K. Partially supported
by Royal Society International Exchanges (Grant No. 170137), NSFC (Grant No. 11771158) and
a Senior Fellowship at Freiburg Institute of Advanced Studies (FRIAS), University of Freiburg.
\texttt{g.liang@warwick.ac.uk} } \and
Thaleia Zariphopoulou\thanks{%
Departments of Mathematics and IROM, The University of Texas at
Austin, U.S.A. and the Oxford-Man Institute, University of Oxford,
U.K. \texttt{zariphop@math.utexas.edu}}}

\begin{document}

\maketitle

\begin{abstract}
We propose an approximation scheme for a class of semilinear
parabolic equations that are convex and coercive in their gradients.
Such equations arise often in pricing and portfolio management in
incomplete markets and, more broadly, are directly connected to the
representation of solutions to backward stochastic differential
equations. The proposed scheme is based on splitting the equation in
two parts, the first corresponding to a linear parabolic equation
and the second to a Hamilton-Jacobi equation. The solutions of these
two equations are approximated using, respectively, the Feynman-Kac
and the Hopf-Lax formulae. We establish the convergence of the
scheme and determine the convergence rate, combining Krylov's
shaking coefficients technique and Barles-Jakobsen's optimal
switching approximation.
\end{abstract}



\pagestyle{myheadings} \thispagestyle{plain} \markboth{Shuo Huang,
Gechun Liang and Thaleia Zariphopoulou}{An approximation scheme for
semilinear parabolic equations}

%%%%%%%%%%%%%%%%%%%%%%%%%%%%%%%%%%%%%%%%%%%%%%%%%%%%%%%%%%%
%%%%%%%%%%%%%%%%%%%%%%%%%%%%%%%%%%%%%%%%%%%%%%%%%%%%%%%%%%

\section{Introduction}

We consider semilinear parabolic equations of the form
\begin{equation}
\label{PDE_1} \left\{\begin{array}{ll} \displaystyle
-\partial_tu-\frac{1}{2}\text{Trace}\left(\sigma\sigma^T(t,x)\partial_{xx}u\right)-b(t,x)\cdot\partial_xu
+H(t,x,\partial_xu)=0&\text{in}\ Q_T;\\
\displaystyle u(T,x)=U(x)&\text{in}\ \mathbb{R}^n,
\end{array}\right.
\end{equation}
where $Q_T=[0,T)\times\mathbb{R}^n$. A key feature is that the
Hamiltonian $H(t,x,p)$ is convex and coercive in $p$. In particular,
the coercivity covers the case that $H$ has quadratic growth in $p$,
a case that corresponds to a rich class of equations in mathematical
finance arising in optimal investment with homothetic risk
preferences (\cite{Hu}), exponential indifference valuation
(\cite{HL, HL1}), entropic risk measures (\cite{CHLZ}) and others.

More broadly, these equations are inherently connected to
(quadratic) backward stochastic differential equations (BSDE), a
central area of stochastic analysis (\cite{DHB} \cite{Peng} and
\cite{Kobylanski}). Specifically, the Hamiltonian $H(t,x,p)$
 is directly related to the BSDE's driver and,
moreover, the solution of (\ref{PDE_1}) yields a functional-form
representation of the processes solving the BSDE.

General existence and uniqueness results can be found, among others
in \cite{Kobylanski} as well as in \cite{Hu}, where BSDE techniques
have been mainly applied. Closed-form solutions can be constructed
only in one-dimensional cases (\cite{Zari}). Furthermore,
approximation schemes have been developed; see \cite{Touzi2} and
\cite{CR} for more references.

Herein, we contribute to further studying problem (\ref{PDE_1}) by
proposing a new approximation scheme. The key idea is to use in an
essential way the \textit{convexity} of the Hamiltonian with respect
to the gradient. This property is natural in all above applications
but it has not been adequately exploited in the existing
approximation studies.

To highlight the main ideas and build intuition, we start with some
preliminary informal arguments, considering for simplicity slightly
simpler equations. To this end, consider the Hamilton-Jacobi (HJ)
equation
\begin{equation}
\label{HJ_equation} \left\{\begin{array}{ll} \displaystyle
-\partial_{t}u+H(\partial_{x}u)=0&\text{in}\ Q_T;\\
\displaystyle u(T,x)=U(x)&\text{in}\ \mathbb{R}^n,
\end{array}\right.
\end{equation}
where the Hamiltonian $H$ is convex and {coercive}, and the terminal
datum $U$ is bounded and Lipschitz continuous. Let $L$ be the
Legendre (convex dual) transform of $H$,
$L(q)=\sup_{p\in\mathbb{R}^n}\{p\cdot q-H(p)\}$. The Fenchel-Moreau
theorem then yields that $H(p)=\sup_{q\in\mathbb{R}^n}\{p\cdot
q-L(q)\}$ and, thus, the HJ equation in (\ref{HJ_equation}) can be
alternatively written as
$$-\partial_{t}u+\sup_{q\in\mathbb{R}^n}\left\{\partial_{x}u\cdot q-L(\partial_{x}u)\right\}=0.$$
Classical arguments from control theory then imply the deterministic
optimal control representation
$$u(t,x)=\inf_{q\in\mathbb{L}^{2}[t,T]}\left[\int_{t}^{T}L(q_{s})ds+U(X_{T}^{t,x;q})\right],$$
with the controlled state equation
$X_{s}^{t,x;q}=x-\int_{t}^{s}q_{u}du$, for $s\in[t,T].$

Hopf and Lax observed that, instead of considering the controls in
$\mathbb{L}^2[t,T]$, it suffices to optimize over the controls
generating geodesic paths of $X^{t,x;q}$, i.e. the controls
$\hat{q}$ such that $X^{t,x;\hat{q}}_T=y$, for any
$y\in\mathbb{R}^n$. Such controls are given by
$\hat{q}_s=\frac{x-y}{T-t}$, for $s\in[t,T]$. The above ``infinite
dimensional'' optimal control problem is thus reduced to the
``finite dimensional'' minimization problem
\begin{equation}\label{Hopf-Lax}
u(t,x)=\inf_{y\in\mathbb{R}^{n}}\left\{(T-t)L(\frac{x-y}{T-t})+U(y)\right\}.\
\ \ (\text{Hopf-Lax formula})
\end{equation}
{There exist several well established algorithms to study this type
of minimization problems (see, for example, \cite{NM} for the
Nelder-Mead simplex algorithm). We also refer to section 3.3.2.b in
\cite{Evans} for the introduction of the Hopf-Lax formula from a
classical calculus of variations perspective.}

Adding a diffusion term to equation (\ref{HJ_equation}) yields the
semilinear parabolic equation
\begin{equation}
\label{split} \left\{\begin{array}{ll}
-\partial_tu-\frac{1}{2}\text{Trace}\left(\sigma\sigma^T(t,x)\partial_{xx}u\right)+H(\partial_xu)=0&\text{in}\  Q_T;\\
u(T,x)=U(x)&\text{in}\ \mathbb{R}^n.
\end{array}\right.
\end{equation}
In analogy to the deterministic case, classical arguments from
control theory imply the stochastic optimal control representation
$$u(t,x)=\inf_{q\in\mathbb{H}^{2}[t,T]}\mathbf{E}\left[\int_{t}^{T}L(q_{s})ds+U(X_{T}^{t,x;q})|\mathcal{F}_t\right],$$
with the controlled state equation
$X_{s}^{t,x;q}=x-\int_{t}^{s}q_{u}du+\int_{t}^{s}\sigma(u,X_{u}^{t,x;q})dW_{u}$,
for $s\in[t,T]$, and $\mathbb{H}^{2}[t,T]$ being the space of
square-integrable progressively measurable processes $q$.

Naturally, due to the stochasticity of the state $X^{t,x;q}$, the
Hopf-Lax formula (\ref{Hopf-Lax}) does not hold for the solution of
problem (\ref{split}). On the other hand, we observe that if we
still choose, as in the deterministic case, controls of the form
$\hat{q}_s=\frac{x-y}{T-t}$, for $y\in\mathbb{R}^n$ and $s\in[t,T]$,
then
$$X_s^{t,x;\hat{q}}=\frac{T-s}{T-t}x+\frac{s-t}{T-t}y+\int_t^s\sigma(u,X_u^{t,x;\hat{q}})dW_u,$$ for $s\in[t,T].$
Therefore, for $T-t=o(1)$, we have $X_{T}^{t,x;\hat{q}}\approx
Y_{T}^{t,{y}}$, where $Y^{t,{y}} $ solves the \emph{uncontrolled}
stochastic differential equation
$$Y_s^{t,y}=y+\int_t^{{s}}\sigma(u,{Y_u^{t,y}})dW_u,$$ for $s\in[t,T].$
In turn, since $y$ is arbitrary, we readily obtain an \emph{upper}
bound of the solution $u(t,x)$ of (\ref{split}), namely,
\begin{equation}\label{upperbound}
u(t,x)\leq
\inf_{y\in\mathbb{R}^n}\left\{(T-t)L(\frac{x-y}{T-t})+\mathbf{E}[U(Y_{T}^{t,y})|\mathcal{F}_t]\right\}.
\end{equation}

Furthermore, the convexity of $H$ yields that $L$ is also convex
and, therefore, for any control process $q\in\mathbb{H}^2[t,T]$, we
deduce that
\begin{align*}
&\ \mathbf{E}\left[\int_{t}^{T}L(q_{s})ds+U(X_{T}^{t,x;q})|\mathcal{F}_t\right]\\
\geq&\
(T-t)L\left(\mathbf{E}\left[\frac{1}{T-t}\int_t^Tq_udu|\mathcal{F}_t\right]\right)+\mathbf{E}[U(X_T^{t,x;q})|\mathcal{F}_t]\\
=&\
(T-t)L\left(\mathbf{E}\left[\frac{x-X_{T}^{t,x;q}+\int_t^{T}\sigma(u,X_u^{t,x;q})dW_u}{T-t}|\mathcal{F}_t\right]\right)+\mathbf{E}[U(X_T^{t,x;q})|\mathcal{F}_t]\\
=&\
(T-t)L\left(\frac{x-\mathbf{E}[X_{T}^{t,x;q}|\mathcal{F}_t]}{T-t}\right)+\mathbf{E}[U(X_T^{t,x;q})|\mathcal{F}_t].
\end{align*}
Therefore, for $T-t=o(1)$, we have $X_{T}^{t,x;q}\approx
Y_{T}^{t,\hat{y}}$, with
$\hat{y}:=\mathbf{E}[X_{T}^{t,x;q}|\mathcal{F}_t]$. Thus, we also
obtain a \emph{lower} bound of the solution $u(t,x)$ of
(\ref{split}), namely,
\begin{equation}\label{lowerbound}
u(t,x)\geq
\inf_{\hat{y}\in\mathbb{R}^n}\left\{(T-t)L(\frac{x-\hat{y}}{T-t})+\mathbf{E}[U(Y_{T}^{t,\hat{y}})|\mathcal{F}_t]\right\}.
\end{equation}
Note that when $\sigma$ degenerates to $0$, {inequalities
(\ref{upperbound}) and (\ref{lowerbound}) give us an equality, which
is precisely the Hopf-Lax formula (\ref{Hopf-Lax}).

We now see how the above ideas can be combined to develop an
approximation scheme for the original problem (\ref{PDE_1}).
Equation (\ref{PDE_1}) can be ``split" into a first-order nonlinear
equation of Hamilton-Jacobi type and a linear parabolic equation.
The solution of the former is represented via the Hopf-Lax formula
and corresponds to the value function of a deterministic control
problem. The solution of the latter corresponds to a conditional
expectation of an uncontrolled diffusion and is given by the
Feynman-Kac formula. The scheme is then naturally based on a
backwards in time recursive combination of the Hopf-Lax and the
Feynman-Kac formula; see (\ref{semigroupequation1}) and
(\ref{semischeme}) for further details.

We establish the convergence of the scheme to the unique (viscosity)
solution of (\ref{PDE_1}) and determine the rate of convergence. We
do this by deriving upper and lower bounds on the approximation
error (Theorems \ref{theorem_error_1} and \ref{theorem_error_2},
respectively). The main tools  come from the \emph{shaking
coefficients technique} introduced by Krylov \cite{Krylov1}
\cite{Krylov} and the \emph{optimal switching approximation}
introduced by Barles and Jakobsen \cite{BJ0} \cite{BJ}.

While various arguments follow from adaptations of these techniques,
the main difficulty is to derive a consistency error estimate. This
is one of the key steps herein and it is precisely where the
convexity of the Hamiltonian with respect to the gradient is used in
an essential way. Specifically, we obtain this estimate by applying
convex duality and using the properties of the optimizers in the
related minimization problems (Proposition \ref{semigroup} (vi)).
Using this estimate and the comparison result for the approximation
scheme (Proposition \ref{schemecomparison}), we in turn derive an
upper bound for the approximation error by perturbing the
coefficients of the equation. The lower bound for the approximation
error is obtained by another layer of approximation of the equation
by using an auxiliary optimal switching system.

Approximation schemes for viscosity solutions were first studied by
Barles and Souganidis \cite{Barles}, who showed that any monotone,
stable and consistent approximation scheme converges to the correct
solution, provided that there exists a comparison principle for the
limiting equation. The corresponding convergence rate had been an
open problem for a long time until late 1990s when Krylov introduced
the {shaking coefficients technique} to construct a sequence of
smooth subsolutions/supersolutions. This technique was further
developed by Barles and Jakobsen in a sequence of papers (see
\cite{BJ1} and \cite{Jakobsen} and more references therein), and
{has} recently been applied to solve various problems (see, among
others, \cite{Erhan} \cite{BPZ} \cite{FTW} and \cite{HL1}).

Krylov's technique depends crucially on the convexity/concavity of
the underlying equation with respect to its terms. As a result,
unless the approximate solution has enough regularity (so one can
interchange the roles of the approximation scheme and the original
equation), the {shaking coefficients technique} only gives either an
upper or a lower bound for the approximation error, but not both. A
further breakthrough was made by Barles and Jakobsen in \cite{BJ0}
and \cite{BJ}, who combined the ideas of optimal switching
approximation of Hamilton-Jacobi-Bellman (HJB) equations (initially
proposed by Evans and Friedman \cite{EF}) with the shaking
coefficients technique. They obtained both upper and lower bounds of
the {error estimate}, but with a lower {convergence rate} due to the
introduction of another approximation layer.

The splitting approach (fractional step, prediction and correction,
etc.) is dated back to Marchuk \cite{Marchuk} in the late 1960s. Its
application to nonlinear PDEs was firstly proposed by Lions and
Mercier \cite{Lions} and has been subsequently used by many others.
For semilinear parabolic equations related to problems in
mathematical finance, splitting methods have been applied by Tourin
\cite{Tourin} (see also more references therein). More recently,
Nadtochiy and Zariphopoulou \cite{NZ} proposed a splitting algorithm
to the marginal HJB equation arising in optimal investment problems
in a stochastic factor model and general utility functions.
Henderson and Liang \cite{HL1} proposed a splitting approach for
utility indifference pricing in a multi-dimensional non-traded
assets model with intertemporal default risk, and established its
convergence rate. Tan \cite{Tan} proposed a splitting method for a
class of fully nonlinear degenerate parabolic PDEs and applied it to
Asian options and commodity trading.

{Finally, we mention that most of the existing algorithms (see,
among others, Howard's finite difference scheme \cite{BMZ}) provide
approximations only at certain time grids. In contrast, the
splitting approximation can be used to approximate the solution at
any time point. Furthermore, since the existing algorithms are often
based on finite difference approximation, the ``curse of
dimensionality" issue arises. We remark that the splitting
approximation itself does not involve finite difference formulation, {as long as one can find an efficient way to compute conditional
expectations, e.g. the multi-level Monte Carlo approach \cite{MLMC_method}, the least squares Monte Carlo approach \cite{LS_method}, the cubature approach \cite{cubature_method}, and etc.} This advantage is also shared by existing BSDE time
discretization algorithms (see, for example, \cite{Touzi2} and
\cite{CR}). However, the commonly used BSDE time discretization
algorithms for (\ref{PDE_1}) require that the Hamiltonian has the
form $H(t,x,\sigma^{tr}(t,x)\partial_xu)$ (see \cite{CR}), which is
not the case herein. Indeed, we do not require the last variable in
the Hamiltonian $H$ to depend on the diffusion coefficient
$\sigma$.}

The paper is organized as follows. In section 2 we introduce the
approximation scheme. In section 3, we prove its convergence rate
using the shaking coefficients technique and optimal switching
approximation. We provide a numerical test in section 4 and conclude
in section 5. Some technical proofs are provided in the appendix.

%%%%%%%%%%%%%%%%%%%%%%%%%%%%%%%%%%%%%%%%%%%%%%%%%%%%%%%%5
%%%%%%%%%%%%%%%%%%%%%%%%%%%%%%%%%%%%%%%%%%%%%%%%%%%%%%%%%

\section{The approximation scheme using the Hopf-Lax formula and splitting} \label{sec-model}

%%%%%%%%%%%%%%%%%%%%%%%%%%%%%%%%%%%%%%%%%%%%%%%%%%%%%%%%%%%%%%%%%%%5

For $T>0$, let $Q_T=[0,T)\times\mathbb{R}^n$. Let also $d$ be a
positive integer and $\delta>0$. For a function $f:Q_T\to
\mathbb{R}^{d}$, we introduce its (semi)norms
$$|f|_{0}:=\sup_{(t,x)\in Q_T}|f(t,x)|,$$
$$[f]_{1,\delta}:=\sup_{\substack{ (t,x),(t',x)\in Q_T \\t\neq t'}}\frac{|f(t,x)-f(t',x)|}{|t-t'|^{\delta}},\ \ \
[f]_{2,\delta}:=\sup_{\substack{ (t,x),(t,x')\in Q_T \\ x\neq x'
}}\frac{|f(t,x)-f(t,x')|}{|x-x'|^{\delta}}.$$ Furthermore,
$[f]_{\delta}:=[f]_{1,\delta/2}+[f]_{2,\delta}$ and
$|f|_{\delta}:=|f|_{0}+[f]_{\delta}$. Similarly, the (semi)norms of
a function $g:\mathbb{R}^n\to\mathbb{R}^{d}$ are defined as
$$|g|_{0}:=\sup_{x\in \mathbb{R}^n}|g(x)|,\ \ \ [g]_{\delta}:=\sup_{\substack{ x,x'\in \mathbb{R}^n \\ x\neq x' }}\frac{|g(x)-g(x')|}{|x-x'|^{\delta}},\ \ \ |g|_{\delta}:=|g|_{0}+[g]_{\delta}.$$
%Note that all the above norms satisfy the usual triangle inequality, and that
%\begin{equation}\label{norm}
%|f|_{\delta}=\sup_{t\in I}|f(t,\cdot)|_{\delta}+[f]_{1,\delta/2}
%\end{equation}

For $S={Q}_T$, $\mathbb{R}^n$ or {$Q_{T}\times \mathbb{R}^{n}$}, we
denote by $\mathcal{C}(S)$ the space of continuous functions on $S$,
and by $\mathcal{C}_b^\delta(S)$ the space of bounded and continuous
functions on $S$ with finite norm $|f|_{\delta}$. We also set
$\mathcal{C}_b^0(S)\equiv\mathcal{C}_b(S)$ and denote by
$\mathcal{C}_b^{\infty}(S)$ the space of smooth functions on $S$
with bounded derivatives of any order.

We  throughout assume the following conditions for equation
(\ref{PDE_1}).

%For $T>0$, we aim to numerically solve the following semilinear PDE
%in the domain $Q_T$:
%\begin{equation}\label{PDE_100}
%-\partial_tu(t,x)-\frac{1}{2}\text{Trace}\left(\sigma(t,x)\sigma^T(t,x)\partial_{xx}u(t,x)\right)-b(t,x)\cdot\partial_xu(t,x)+H(t,x,\partial_xu(t,x))=0,
%\end{equation}
%with terminal condition
%\begin{equation}\label{terminal00}
%u(T,x)=U(x).
%\end{equation}

\begin{assumption}\label{data assumption}

(i) The diffusion coefficient $\sigma\in\mathcal{C}_{b}^{1}(Q_T)$,
the drift coefficient $b\in\mathcal{C}_b^1(Q_T)$, and the terminal
datum $U\in\mathcal{C}_b^1(\mathbb{R}^n)$ have norms $|\sigma|_{1},
|b|_{1}, |U|_{1}\le M$, for some $M>0$.

(ii) The Hamiltonian $H(t,x,p)\in \mathcal{C}(Q_{T}\times\mathbb{R}^n)$ is %first order differentiable,
convex in p, and satisfies the coercivity condition
$$\lim_{|p|\rightarrow\infty}\frac{H(t,x,p)}{|p|}=\infty,$$
uniformly in $(t,x)\in Q_{T}$. Moreover, for every $p$,
$[H(\cdot,\cdot,p)]_{1}\le M$, and there exist two locally bounded
functions $H^{*}$ and $H_{*}:\mathbb{R}^{n}\to\mathbb{R}$ such that
$$H_{*}(p)=\inf_{(t,x)\in Q_{T}}H(t,x,p)\ \ \text{and}\ \
H^{*}(p)=\sup_{(t,x)\in Q_{T}}H(t,x,p).$$
\end{assumption}

Under the above assumptions, we have the following existence,
uniqueness and regularity results for equation (\ref{PDE_1}). Their
proofs are provided in Appendix A.

\begin{proposition}\label{solutionproperty} Suppose that Assumption \ref{data assumption} is
satisfied. Then, there exists a unique viscosity solution
$u\in\mathcal{C}_b^1(\bar{Q}_T)$ of equation (\ref{PDE_1}), with
$|u|_{1}\le C$, for some constant $C$ depending only on $M$ and $T$.
\end{proposition}

%%%%%%%%%%%%%%%%%%%%%%%%%%%%%%%%%%%%%%%%%%%%%%%%%%%%%%%%%%%%%%%%%%%%5
\subsection{The backward operator $\mathbf{S}_t(\Delta)$}

Using that $H(t,x,p)$ is convex in $p$, we define its Legendre
(convex dual) transform
$L:Q_{T}\times\mathbb{R}^n\rightarrow\mathbb{R}$, given by
\begin{equation}\label{L}
L(t,x,q):=\sup_{p\in\mathbb{R}^n}\{p\cdot q-H(t,x,p)\}.
\end{equation}

For any $t$ and $\Delta$ with $0\leq t<t+\Delta \leq T$, and any
$\phi\in\mathcal{C}_b(\mathbb{R}^{n})$, we introduce the
\emph{backward operator}
$\mathbf{S}_{t}(\Delta):\mathcal{C}_b(\mathbb{R}^{n})\to\mathcal{C}_b(\mathbb{R}^{n})$,
\begin{equation}\label{semigroupequation1}
\left\{
\begin{array}{ll}
\displaystyle \mathbf{S}_{t}(\Delta)\phi(x)=\min_{y\in\mathbb{R}^n} \left\{\Delta L\left(t,x,\frac{x-y}{\Delta}\right)+\mathbf{E}[\phi(Y_{t+\Delta}^{t,y})|\mathcal{F}_t]\right\},& x\in\mathbb{R}^n,\\
\displaystyle Y_{s}^{t,y}=y+ b(t,y)(s-t)+\sigma(t,y)(W_{s}-W_t),&
s\in[t,t+\Delta],
\end{array}
\right.
\end{equation}
on a filtered probability space
$(\Omega,\mathcal{F},\{\mathcal{F}_t\}_{t\geq 0},\mathbf{P})$, where
$W$ is an $n$-dimensional Brownian motion with its augmented
filtration $\{\mathcal{F}_t\}_{t\geq 0}$.\\

We start with some auxiliary properties of $H$ and $L$.

\begin{proposition}\label{transformlemma}
Suppose that Assumption \ref{data assumption} (ii) is satisfied.
Then, the following assertions hold:

(i) $H$ is the Legendre transform of $L$, i.e.
$H(t,x,p)=\sup_{q\in\mathbb{R}^n}\{p\cdot q-L(t,x,q)\},$ for
$(t,x)\in Q_T$.

(ii) The functions  $$L_{*}(q):=\sup_{p\in\mathbb{R}^n}\{p\cdot
q-H^{*}(p)\}\ \ \text{and}\ \
L^{*}(q):=\sup_{p\in\mathbb{R}^n}\{p\cdot q-H_{*}(p)\}$$ are locally
bounded and satisfy, for $(t,x)\in Q_{T}$, $L_{*}(q)\leq
L(t,x,q)\leq L^{*}(q).$

(iii) For $(t,x)\in Q_T$, $L(t,x,q)$ is convex in $q$ with
$[L(\cdot,\cdot,q)]_{1}\le 2M$. Furthermore, it satisfies the
coercivity condition
$$\lim_{|q|\rightarrow\infty}\frac{L(t,x,q)}{|q|}=\infty,$$
uniformly in $(t,x)\in Q_{T}$.

(iv) For each $(t,x)\in Q_{T}$ and $p, q\in\mathbb{R}^{n}$, there
exist $p^{*}, q^{*}\in\mathbb{R}^{n}$ such that
$$L(t,x,q)=q\cdot p^{*}-H(t,x,p^{*})\ \ \text{and}\ \ H(t,x,p)=p\cdot q^{*}-L(t,x,q^{*}).$$
Furthermore, $|p^{*}|\leq \xi(|q|)$ and $|q^{*}|\leq \xi(|p|),$ for
some real-valued increasing function $\xi(\cdot)$ independent of
$(t,x)$.
\end{proposition}

\begin{proof}
Parts (i) and (ii) are immediate and, thus, we only prove (iii) and
(iv).

(iii) For fixed $(t,x)\in Q_{T}$, $q_{1},q_{2}\in\mathbb{R}^{n}$ and
$\lambda\in[0,1]$, we have
\begin{align*}
L(t,x,\lambda q_{1}+(1-\lambda)q_{2}) = &\   \sup_{p\in\mathbb{R}^n}\{(\lambda q_{1}+(1-\lambda)q_{2})\cdot p-H(t,x,p)\} \\
\le &\
 \lambda\sup_{p\in\mathbb{R}^n}\{q_{1}\cdot p-H(t,x,p)\}
 +
 (1-\lambda)\sup_{p\in\mathbb{R}^n}\{q_{2}\cdot p-H(t,x,p)\} \\
 = &\
 \lambda L(t,x,q_{1})+(1-\lambda) L(t,x,q_{2}).
\end{align*}
From the definition of $L$,  we further have, for any
$q\in\mathbb{R}^n$,
\begin{align*}
[L(\cdot,\cdot,q)]_{1}= &\  [L(\cdot,\cdot,q)]_{1,1/2}+[L(\cdot,\cdot,q)]_{2,1} \\
 \le &\
 \sup_{p}\{[H(\cdot,\cdot,p)]_{1,1/2}\}+\sup_{p}\{[H(\cdot,\cdot,p)]_{2,1}\}
 \le
 2M.
\end{align*}
Next, for any $K>0$, we deduce, by setting $p=K\frac{q}{|q|}$, that
\begin{equation*}
L(t,x,q)\ge q\cdot K\frac{q}{|q|}-H(t,x,K\frac{q}{|q|})  \ge
K|q|-\sup_{r\in B(0,K)}H^{*}(r).
\end{equation*}
Dividing both sides by $|q|$ and sending $|q|\to\infty$, the
coercivity condition for $L$ follows.

(iv) From (i) and (ii), we deduce that $L$ and $H$ are symmetric to
each other and, thus, we only establish the assertions for $L$. To
this end, for each $(t,x)\in Q_{T}$, we obtain, by setting $p=0$ in
(\ref{L}), that $L(t,x,q)\ge -H(t,x,0).$ Therefore, it suffices to
find a real-valued increasing function, say $\xi(\cdot)$, such that,
if $|p|>\xi(|q|)$, then
$$p\cdot q-H(t,x,p)<-H(t,x,0).$$
Indeed, it follows from Assumption \ref{data assumption} (ii) that
there exists a real-valued increasing function, say $K_{H}(y)$, such
that, for any $(t,x)\in Q_{T}$ and $|p|\ge K_{H}(y)$, we have
$\frac{H(t,x,p)}{|p|}\ge y$. Setting
$\xi(x):=\max\{K_{H}(|H^{*}(0)|+x),1\}$, we deduce that, for
$|p|>\xi(|q|)$,
\begin{align*}
p\cdot q-H(t,x,p) \le &\ |p|(|q|-\frac{H(t,x,p)}{|p|})<
|q|-(|H^{*}(0)|+|q|) \le -H(t,x,0),
\end{align*}
and we easily conclude.
\end{proof}\\

Next, we show that the minimum in (\ref{semigroupequation1}) is
actually achieved, i.e. for any
$\phi\in\mathcal{C}_{b}(\mathbb{R}^{n})$, there always exists an
associated minimizer $y^{*}$.

\begin{proposition}\label{minimiser}
Suppose that Assumption \ref{data assumption} is satisfied. Then,
for each $t$ and $\Delta$ with $0\leq t<t+\Delta \leq T$,
$x\in\mathbb{R}^{n}$ and $\phi\in\mathcal{C}_b(\mathbb{R}^{n})$,
there exists a minimizer $y^{*}\in\mathbb{R}^{n}$ such that
$$\mathbf{S}_{t}(\Delta)\phi(x)=\Delta L\left(t,x,\frac{x-y^{*}}{\Delta}\right)+\mathbf{E}[\phi(Y^{t,y^{*}}_{t+\Delta})|\mathcal{F}_t].$$
Moreover, there exists a constant $C>0$, depending only on $M$ and
$T$, such that
\begin{equation}\label{inequ}
\left|\frac{x-y^{*}}{\Delta}\right|\le \xi(C[\phi]_{1}),
\end{equation}
%\mathbf{R}(C[\phi]_1),$$
for some real-valued increasing function $\xi(\cdot)$ independent of
$(t,x)$.
\end{proposition}

\begin{proof}
Let $q=\frac{x-y}{\Delta}$. Then $|q|\to\infty$ as $|y|\to\infty$.
In turn, from Proposition \ref{transformlemma} (iii), we deduce
that, as $|y|\rightarrow\infty$,
$$ \Delta L\left(t,x,\frac{x-y}{\Delta}\right)+\mathbf{E}[\phi(Y^{t,y}_{t+\Delta})|\mathcal{F}_t]=|x-y|\frac{L(t,x,q)}{|q|}+\mathbf{E}[\phi(Y^{t,y}_{t+\Delta})|\mathcal{F}_t]\rightarrow\infty.$$
Furthermore, using that the mapping $y\mapsto\Delta
L(t,x,\frac{x-y}{\Delta})+\mathbf{E}[\phi(Y^{t,y}_{t+\Delta})|\mathcal{F}_t]$
is continuous, we deduce that it must admit a minimizer
$y^{*}\in\mathbb{R}^{n}$.

Next, we prove inequality (\ref{inequ}). For
$\phi\in\mathcal{C}^{1}_b(\mathbb{R}^{n})$, following the same
reasoning as in the proof of Proposition \ref{transformlemma} (iv),
it suffices to find a real-valued increasing function $\xi(\cdot)$
such that
\begin{equation}\label{minimiserproof}
\Delta L\left(t,x,q\right)+\mathbf{E}[\phi(Y^{t,x-\Delta
q}_{t+\Delta})|\mathcal{F}_t]>\Delta
L\left(t,x,0\right)+\mathbf{E}[\phi(Y^{t,x}_{t+\Delta})|\mathcal{F}_t],
\end{equation}
if $|q|>\xi(C[\phi]_{1})$, for some constant $C>0$ depending only on
$M$ and $T$. To prove this, note that Assumption \ref{data
assumption} (i) on the coefficients $\sigma$ and $b$ implies that
\begin{align}\label{est1}
\mathbf{E}[\phi(Y^{t,x}_{t+\Delta})|\mathcal{F}_t]-\mathbf{E}[\phi(Y^{t,x-\Delta q}_{t+\Delta})|\mathcal{F}_t]&\ \le [\phi]_{1}\mathbf{E}\left[\left|Y^{t,x}_{t+\Delta}-Y^{t,x-\Delta q}_{t+\Delta}\right||\mathcal{F}_t\right] \nonumber \\
 &\ \le C[\phi]_{1}\Delta |q|.
\end{align}
On the other hand, from Proposition \ref{transformlemma} (iv), there
exists a real-valued increasing function, say $K_{L}(y)$, such that,
for any $(t,x)\in Q_{T}$ and $|q|\ge K_{L}(y)$, we have
$\frac{L(t,x,q)}{|q|}\ge y$. Setting
$\xi(x):=\max\{K_{L}(|L^{*}(0)|+x),1\}$, we deduce that, for
$|q|>\xi(C[\phi]_{1})$,
$$\frac{L(t,x,q)}{|q|}>|L^{*}(0)|+C[\phi]_{1}\ge \frac{L^{*}(0)}{|q|}+C[\phi]_{1}\ge \frac{L(t,x,0)}{|q|}+C[\phi]_{1}.$$
Using the above inequality, together with (\ref{est1}), we obtain
(\ref{minimiserproof}). Finally, the case $[\phi]_{1}=\infty$
follows trivially.
\end{proof}\\

Next, we derive some key properties of the backward operator
$\mathbf{S}_{t}(\Delta)$.

\begin{proposition}\label{semigroup} Suppose that Assumption \ref{data
assumption} is satisfied. Then, for each $t$ and $\Delta$ with
$0\leq t<t+\Delta \leq T$, the operator $\mathbf{S}_{t}(\Delta)$ has
the following properties:

(i) (Constant preserving) For any
$\phi\in\mathcal{C}_b(\mathbb{R}^{n})$ and $c\in\mathbb{R}$,
$$\mathbf{S}_{t}(\Delta)(\phi+c)=\mathbf{S}_{t}(\Delta)\phi+c.$$

(ii) (Monotonicity) For any
$\phi,\psi\in\mathcal{C}_b(\mathbb{R}^{n})$ with $\phi\geq \psi$,
$$\mathbf{S}_{t}(\Delta)\phi\geq\mathbf{S}_{t}(\Delta)\psi.$$

(iii) (Concavity) For any $\phi\in\mathcal{C}_b(\mathbb{R}^n)$,
$\mathbf{S}_{t}(\Delta)\phi$ is concave in $\phi$.

(iv) (Stability) For any $\phi\in\mathcal{C}_b(\mathbb{R}^{n})$,
$$|\mathbf{S}_{t}(\Delta)\phi|_{0}\le C\Delta+|\phi|_{0},$$
where $C=\max\left\{|L^{*}(0)|,|H^{*}(0)|\right\}$, with $L^*$ and
$H^*$ as in Proposition \ref{transformlemma} (ii) and Assumption
\ref{data assumption} (ii). Therefore, the operator
$\mathbf{S}_{t}(\Delta)$ is indeed a mapping from
$\mathcal{C}_b(\mathbb{R}^{n})$ to $\mathcal{C}_b(\mathbb{R}^{n})$.

(v) For any $\phi\in\mathcal{C}^{1}_b(\mathbb{R}^{n})$, there exists
a constant $C$ depending only on $[\phi]_1$, $M$ and $T$, such that
$$|\mathbf{S}_{t}(\Delta)\phi-\phi|_{0}\le C\sqrt{\Delta}.$$

(vi) For any $\phi\in\mathcal{C}_b^{\infty}(\mathbb{R}^{n})$, define
\begin{equation}\label{operator_E}
\mathcal{E}(t,\Delta,\phi):=
\left|\frac{\phi-\mathbf{S}_{t}(\Delta)\phi}{\Delta}-{\mathbf{L}}_{t}\phi\right|_{0},
\end{equation}
where the operator $\mathbf{L}_{t}$ is given by
$$\mathbf{L}_{t}\phi(x)=-\frac{1}{2}\text{Trace}\left(\sigma\sigma^T(t,x)\partial_{xx}\phi(x)\right)-b(t,x)\cdot\partial_x\phi(x)+H(t,x,\partial_x\phi(x)).$$
Then,
\begin{equation*}
\mathcal{E}(t,\Delta,\phi)\leq \ C\Delta
\left(|\partial_{xxxx}\phi|_0+\mathcal{R}(\phi)\right),
\end{equation*}
where the constant $C$ depends only on $[\phi]_{1}$, $M$ and $T$,
and $\mathcal{R}(\phi)$ represents the ``insignificant'' terms
containing the derivatives of $\phi$ up to third order.
\end{proposition}

\begin{proof} Parts (i)-(iii) are immediate. We only prove (iv)-(vi) and, in particular, for the case $n=1$, since the general
case follows along similar albeit more complicated arguments. %To prove (iv), we need only to prove for any $x\in\mathbb{R}^n$, $\frac{\phi(x)-\mathbf{S}_{t}(\Delta)\phi(x)}{\Delta}\rightarrow{\mathbf{L}}_{t}\phi(x)$, as $\Delta\rightarrow 0$.

(iv) Choosing $y=x$ in (\ref{semigroupequation1}) gives
\begin{align}\label{estimate_1}
\mathbf{S}_{t}(\Delta)\phi(x) \leq \Delta L^{*}(0)+|\phi|_{0}.
\end{align}
It follows from the definition of $L_{*}$ in Proposition
\ref{transformlemma} (ii) that $L_{*}(q)\ge -H^{*}(0) \ge
-|H^{*}(0)|$, for $q\in\mathbb{R}^{n}$. In turn, Proposition
\ref{minimiser} further yields
\begin{align}\label{estimate_2}
\mathbf{S}_{t}(\Delta)\phi(x)= &\
\Delta L(t,x,\frac{x-y^{*}}{\Delta})+\mathbf{E}[\phi(Y^{t,y^{*}}_{t+\Delta})|\mathcal{F}_t] \notag\\
\ge &\
\Delta L_{*}(\frac{x-y^{*}}{\Delta})-|\phi|_{0} \notag\\
\ge &\ -\Delta |H^{*}(0)|-|\phi|_{0}.
\end{align}
The assertion then follows by combining (\ref{estimate_1}) and
(\ref{estimate_2}).

(v) From Proposition \ref{transformlemma} (ii) and Proposition
\ref{minimiser}, we deduce that
\begin{align*}
\left|\mathbf{S}_{t}(\Delta)\phi(x)-\phi(x)\right|= &\
\left|\Delta L(t,x,\frac{x-y^{*}}{\Delta})+\mathbf{E}[\phi(Y^{t,y^{*}}_{t+\Delta})-\phi(x)|\mathcal{F}_{t}]\right|\\
\le &\
\Delta\max\left\{|L^{*}(\frac{x-y^{*}}{\Delta})|, |L_{*}(\frac{x-y^{*}}{\Delta})|\right\}+[\phi]_{1}\mathbf{E}\left[\left|Y^{t,y^{*}}_{t+\Delta}-x\right||\mathcal{F}_{t}\right]\\
\le &\ C\Delta+(C\Delta+M\Delta+CM\sqrt{\Delta})[\phi]_{1} \le
C\sqrt{\Delta},
\end{align*}
where the constant $C$ depends only on $[\phi]_1$, $M$ and $T$.

(vi) For $(t,x)\in[0,T-\Delta]\times\mathbb{R}$, let
$q^*\in\mathbb{R}$ be such that
$$H(t,x,\partial_x\phi(x))=\max_{q\in\mathbb{R}}\{q\partial_x\phi(x)-L(t,x,q)\}=q^*\partial_x\phi(x)-L(t,x,q^*).$$
From Proposition \ref{transformlemma} (iv), we have $|q^{*}|\le
\xi(|\partial_{x}\phi(x)|)\le C$, where the constant $C$ depends
only on $[\phi]_{1}$, $M$ and $T$.

Choosing $y=x-\Delta q^*$ in (\ref{semigroupequation1}) and applying
It\^o's formula to $\phi(Y^{t,x-\Delta q^{*}}_{t+\Delta})$ yield
\begin{align*}
  &\  \phi(x)-\mathbf{S}_{t}(\Delta)\phi(x)-\Delta{\mathbf{L}}_{t}\phi(x)\\
\ge &\
\phi(x)-\Delta L(t,x,q^*)-\phi(x-\Delta q^*)-\mathbf{E}[\phi(Y^{t,x-\Delta q^*}_{t+\Delta})-\phi(x-\Delta q^*)|\mathcal{F}_t] -\Delta{\mathbf{L}}_{t}\phi(x)\\
= &\
\left(\phi(x)-\phi(x-\Delta q^*)-\Delta q^*\partial_{x}\phi(x)\right)\\
&-
\left(\mathbf{E}\left[\int_t^{t+\Delta}\left(b(t,y)\partial_x\phi(Y_s^{t,x-\Delta
q^*})+\frac{1}{2}{|\sigma(t,y)|^2\partial_{xx}\phi(Y_s^{t,x-\Delta
q^*})}\right)ds|\mathcal{F}_t\right]\right.\\
&-\left.\Delta
b(t,x)\partial_x\phi(x)-\frac{1}{2}{\Delta|\sigma(t,x)|^2\partial_{xx}\phi(x)}
\right):=(I)-(II).
\end{align*}

Next, we obtain a lower and an upper bound for terms (I) and (II),
respectively. To this end, Taylor's expansion yields
\begin{align}\label{estimate_3}
  &\  \phi(x)-\phi(x-\Delta q^*)-\Delta q^*\partial_{x}\phi(x)\notag\\
= &\
\int_{x-\Delta q^{*}}^{x}\left(\partial_{x}\phi(x)-\int_{s}^{x}\partial_{xx}\phi(u)du\right)ds-\Delta q^{*}\partial_{x}\phi(x)\notag\\
\ge &\ -C\Delta^{2}|\partial_{xx}\phi|_{0}.
\end{align}

For term (II), applying It\^o's formula to
$\partial_x\phi(Y_s^{t,x-\Delta q^{*}})$ and
$\partial_{xx}\phi(Y_s^{t,x-\Delta q^{*}})$ gives
\begin{align*}
 &\ \mathbf{E}\left[\partial_{x}\phi(Y_s^{t,x-\Delta q^{*}})|\mathcal{F}_t\right]\\
 = &\ \partial_{x}\phi(y)+\int_{t}^{s}\mathbf{E}\left[b(t,y)\partial_{xx}\phi(Y_u^{t,x-\Delta q^{*}})+\frac12
 |\sigma(t,y)|^{2}\partial_{xxx}\phi(Y_{u}^{t,x-\Delta q^{*}})|\mathcal{F}_t\right]du,
\end{align*}
and
\begin{align*}
 &\ \mathbf{E}\left[\partial_{xx}\phi(Y_s^{t,x-\Delta q^{*}})|\mathcal{F}_t\right]\\
 = &\ \partial_{xx}\phi(y)+\int_{t}^{s}\mathbf{E}\left[b(t,y)\partial_{xxx}\phi(Y_u^{t,x-\Delta q^{*}})+\frac12
 |\sigma(t,y)|^{2}\partial_{xxxx}\phi(Y_{u}^{t,x-\Delta q^{*}})|\mathcal{F}_t\right]du.
\end{align*}

Keeping the terms involving the derivatives of $\phi$ and using
Assumption \ref{data assumption} on $b$ and $\sigma$, we further
have
\begin{align}\label{estimate_4}
  &\  \mathbf{E}\left[\int_t^{t+\Delta}\left(b(t,y)\partial_x\phi(Y_s^{t,x-\Delta q^{*}})+\frac{1}{2}{|\sigma(t,y)|^2\partial_{xx}\phi(Y_s^{t,x-\Delta q^{*}})}\right)ds|\mathcal{F}_t\right]\notag\\
  &\ -\Delta b(t,x)\partial_x\phi(x)-\frac{1}{2}{\Delta|\sigma(t,x)|^2\partial_{xx}\phi(x)} \notag\\
 \le &\
 C\Delta^{2}(|\partial_{x}\phi|_{0}+|\partial_{xx}\phi|_{0}+|\partial_{xxx}\phi|_{0}+|\partial_{xxxx}\phi|_{0}).
\end{align}

In turn, combining estimates (\ref{estimate_3}) and
(\ref{estimate_4}) above, we deduce that
$$ \frac{ \phi(x)-\mathbf{S}_{t}(\Delta)\phi(x)}{\Delta}-{\mathbf{L}_{t}}\phi(x)\ge  -C\Delta(|\partial_{x}\phi|_{0}+|\partial_{xx}\phi|_{0}+|\partial_{xxx}\phi|_{0}+|\partial_{xxxx}\phi|_{0}),$$
where the constant $C$ depends only on $[\phi]_{1}$, $M$ and $T$.

To prove the reverse inequality, we work as follows. For
$(t,x)\in[0,T-\Delta]\times\mathbb{R}$, let $y^{*}\in\mathbb{R}$ be
the minimizer in (\ref{semigroupequation1}) and set
$p^{*}:=\frac{x-y^{*}}{\Delta}$. Then, we deduce from Proposition
\ref{minimiser} that $|p^{*}|\leq C$, where the constant $C$ depends
only on $[\phi]_{1}$, $M$ and $T$. In turn, similar calculations as
above yield
\begin{align*}
  &\  \phi(x)-\mathbf{S}_{t}(\Delta)\phi(x)-\Delta{\mathbf{L}}_{t}\phi(x)\\
= &\
\phi(x)-\Delta L(t,x,p^{*})-\phi(x-\Delta p^{*})-\mathbf{E}[\phi(Y^{t,x-\Delta p^{*}}_{t+\Delta})-\phi(x-\Delta p^{*})|\mathcal{F}_t] -\Delta{\mathbf{L}}_{t}\phi(x)\\
= &\
\Delta \left(p^{*}\partial_{x}\phi(x)-L(t,x,p^{*})\right)-\Delta H(t,x,\partial_{x}\phi(x))-\int_{x}^{x-\Delta p^{*}}\left(\int_{x}^{s}\partial_{xx}\phi(u)du\right)ds\\
&- \left(\mathbf{E}\left[\int_t^{t+\Delta}\left(b(t,y)\partial_x\phi(Y_s^{t,x-\Delta p^*})+\frac{1}{2}{|\sigma(t,y)|^2}\partial_{xx}\phi(Y_s^{t,x-\Delta p^*})\right)ds|\mathcal{F}_t\right]\right.\\
&- \left.\Delta b(t,x)\partial_x\phi(x)-\frac12\Delta|\sigma(t,x)|^2\partial_{xx}\phi(x)\right)\\
\le &\
 C\Delta^{2}(|\partial_{x}\phi|_{0}+|\partial_{xx}\phi|_{0}+|\partial_{xxx}\phi|_{0}+|\partial_{xxxx}\phi|_{0}),
\end{align*}
for some constant $C$ depending only on $[\phi]_{1}$, $M$ and $T$.
We easily conclude.
\end{proof}

\subsection{The approximation scheme}\label{sebsection: splitting}

We present the approximation scheme for equation (\ref{PDE_1}).
%Let $N$ be a positive integer and set $\Delta:=\frac{T}{N}$. Then,
For $(t,x)\in\bar{Q}_{T-\Delta}$, we introduce the iterative
algorithm
\begin{equation}\label{splitting}
u^{\Delta}(t,x)=\mathbf{S}_{t}(\Delta)u^{\Delta}(t+\Delta,\cdot)(x),
\end{equation}
with $u^{\Delta}(T,\cdot)=U(\cdot)$ and $\mathbf{S}_t(\Delta)$
defined in (\ref{semigroupequation1}). The values between $T-\Delta$
and $T$ are obtained by a standard linear interpolation.

Specifically, the approximation scheme is given by
\begin{eqnarray}\label{semischeme}
\left\{
\begin{array}{ll}
\displaystyle S(\Delta, t,x,u^\Delta(t,x),u^\Delta(t+\Delta,\cdot))=0& \text{in}\ \bar{Q}_{T-\Delta};\\
\displaystyle u^\Delta(t,x)=g^\Delta(t,x)&\text{in}\
\bar{Q}_{T}\backslash\bar{Q}_{T-\Delta},
\end{array}
\right.
\end{eqnarray}
where
$S:\mathbb{R}^+\times\bar{Q}_{T-\Delta}\times\mathbb{R}\times\mathcal{C}_b(\mathbb{R}^n)\rightarrow\mathbb{R}$
and $g^\Delta:\bar{Q}_{T}\backslash\bar{Q}_{T-\Delta}\to\mathbb{R}$
are defined, respectively, by
\begin{equation}
S(\Delta,t,x,p,v)=\frac{p-\mathbf{S}_{t}(\Delta)v(x)}{\Delta}
\end{equation}and
\begin{equation}\label{gdelta}
g^\Delta(t,x)=\omega_1(t)U(x)+\omega_2(t)\mathbf{S}_{T-\Delta}(\Delta)U(x),
\end{equation}
with $\omega_1(t)=(t+\Delta-T)/\Delta$ and
$\omega_2(t)=(T-t)/\Delta$ being the linear interpolation weights.

Note that when $T-\Delta<t\le T$, the approximation term $g^\Delta$
corresponds to the usual linear interpolation between $T-\Delta$ and
$T$. When $t= T-\Delta$, we have $\omega_1(t)=0$ and $\omega_2(t)=1$
and, thus, $g^\Delta(T-\Delta,x)=u^\Delta(T-\Delta,x)$.

We first prove the well-posedness of the approximation scheme
(\ref{semischeme}).

\begin{lemma} %(Convergence of approximation scheme)
\label{semigrouptheorem} Suppose that Assumption 2.1 is satisfied.
Then, the approximation scheme (\ref{semischeme}) admits a unique
solution $u^\Delta\in\mathcal{C}_b(\bar{Q}_{T})$, with
$|u^{\Delta}|_{0}\leq C$, where the constant $C$ depends only on $M$
and $T$.
\end{lemma}
\begin{proof}
By the stability property (iv) in Proposition \ref{semigroup}, we
have that $\mathbf{S}_{t}(\Delta)\phi$ is uniformly bounded if so is
$\phi$. Therefore, equation (\ref{semischeme}) is always well
defined in $\bar{Q}_{T-\Delta}$, which yields the existence and
uniqueness of the solution $u^\Delta$. Furthermore, for $0\leq t\leq
T-\Delta$, $|u^{\Delta}(t,\cdot)|_{0}\leq
C\Delta+|u^{\Delta}(t+\Delta,\cdot)|_{0}.$ By backward induction and
the definition of $g^{\Delta}$ in (\ref{semischeme}), we conclude
that
$$|u^{\Delta}|_{0}\leq CT+\sup_{t\in(T-\Delta,T]}|g^{\Delta}(t,\cdot)|_{0}\leq C,$$
where the constant $C$ depends only on $M$ and $T$.
\end{proof}\\

\begin{lemma}\label{errorsmall}
Suppose that Assumption \ref{data assumption} holds. Let
$u^{\Delta}\in\mathcal{C}_b(\bar{Q}_{T})$ satisfy the approximation
scheme (\ref{semischeme}) and $u\in\mathcal{C}_b^1(\bar{Q}_{T})$ be
the unique viscosity solution of equation (\ref{PDE_1}). Then, there
exists a constant $C$, depending only on $M$ and $T$, such that
\begin{equation}\label{estimate_final_interval}
|u-u^{\Delta}|\leq C\sqrt{\Delta}\ \ \text{in}\
\bar{Q}_{T}\backslash \bar{Q}_{T-\Delta}.
\end{equation}
\end{lemma}
\begin{proof}
From (\ref{semischeme}), we
 have, for $(t,x)\in \bar{Q}_{T}\backslash \bar{Q}_{T-\Delta}$,
\begin{align*}
|u(t,x)-u^{\Delta}(t,x)|= &\ |u(t,x)-g^{\Delta}(t,x)|\\
=&\ |u(t,x)-u(T,x)+\omega_{2}(t)(U(x)-\mathbf{S}_{T-\Delta}(\Delta)U(x))| \\
 \leq &\
 |u(t,x)-u(T,x)|+|U(x)-\mathbf{S}_{T-\Delta}(\Delta)U(x)| \\
 \leq &\
 C(\sqrt{|T-t|}+\sqrt{\Delta}) \leq C\sqrt{\Delta},
\end{align*}
where the second to last inequality follows from the regularity
property of the solution $u$ (cf. Proposition
\ref{solutionproperty}) and property (\emph{v}) of the operator
$\mathbf{S}_{t}(\Delta)$ (cf. Proposition \ref{semigroup}).
\end{proof}\\

Using the properties of $\mathbf{S}_{t}(\Delta)$ established in
Proposition \ref{semigroup}, we next obtain the following key
properties of the approximation scheme (\ref{semischeme}).

\begin{proposition}\label{scheme property}
Suppose that Assumption \ref{data assumption} is satisfied. Then,
for each $t$ and $\Delta$ with $0\leq t<t+\Delta \leq T$,
$x\in\mathbb{R}^{n}$, $p\in\mathbb{R}$ and
$v\in\mathcal{C}_{b}(\mathbb{R}^{n})$, the approximation scheme
$S(\Delta,t,x,p,v)$ has the following properties:

(i) (Constant preserving) For any $c\in\mathbb{R}$,
$$S(\Delta,t,x,p+c,v+c)=S(\Delta,t,x,p,v).$$

(ii) (Monotonicity) For any  $u\in\mathcal{C}_{b}(\mathbb{R}^{n})$
with $u\le v$,
$$S(\Delta,t,x,p,u)\ge S(\Delta,t,x,p,v).$$

(iii) (Convexity) $S(\Delta,t,x,p,v)$ is convex in $p$ and $v$.

(iv) (Consistency) For any
$\phi\in\mathcal{C}_b^{\infty}(\bar{Q}_{T})$, there exists a
constant $C$, depending only on $[\phi]_{2,1}$, $M$ and $T$, such
that
\begin{align}\label{consistant_error}
 &\ |-\partial_t\phi(t,x)+\mathbf{L}_t\phi(t,x)-S(\Delta,t,x,\phi(t,x),\phi(t+\Delta,\cdot))|\notag\\
 \le  &\
 C\Delta\left(|\partial_{tt}\phi|_{0}+|\partial_{xxxx}\phi|_0+|\partial_{xxt}\phi|_{0}+\mathcal{R}(\phi)\right).
\end{align}
\end{proposition}

\begin{proof}
Parts (i)-(iii) follow easily from Proposition \ref{semigroup}, so
we only prove (iv). To this end, we split the consistency error into
three parts. Specifically,
\begin{align*}
&|-\partial_t\phi(t,x)+\mathbf{L}_t\phi(t,x)-S(\Delta,t,x,\phi(t,x),\phi(t+\Delta,\cdot))|\\
\leq &\ {\mathcal{E}(t,\Delta,\phi(t+\Delta,\cdot))} +
{|\phi(t+\Delta,x)-\phi(t,x)-\Delta\partial_t\phi(t,x)|\Delta^{-1}}\\
&+
{|\mathbf{L}_t\phi(t,x)-\mathbf{L}_t\phi(t+\Delta,x)|}:=(I)+(II)+(III),
\end{align*}
where $\mathcal{E}$ was defined in (\ref{operator_E}). For term (I),
Proposition \ref{semigroup} (vi) yields
\begin{align}\label{estimate8}
\mathcal{E}(t,\Delta,\phi(t+\Delta,\cdot)) \leq &\ C\Delta\left(|\partial_{xxxx}\phi(t+\Delta,\cdot)|_{0}+\mathcal{R}(\phi(t+\Delta,\cdot))\right)\notag\\
 \leq &\
 C\Delta\left(|\partial_{xxxx}\phi|_{0}+\mathcal{R}(\phi)\right),
\end{align}
for some constant $C$ depending only on $[\phi]_{2,1}$, $M$ and $T$.
For term (II), Taylor's expansion gives
\begin{align}\label{estimate9}
&\
|\phi(t+\Delta,x)-\phi(t,x)-\Delta\partial_t\phi(t,x)|\Delta^{-1}\notag\\
\leq&\
|\int_{t}^{t+\Delta}\left(\partial_{t}\phi(t,x)-\int_v^{t}\partial_{tt}\phi(u,x)du\right)dv-
\Delta\partial_t\phi(t,x)|\Delta^{-1}\notag\\
\leq&\ \Delta|\partial_{tt}\phi|_0.
\end{align}
Finally, for term (III), we have from Assumption \ref{data
assumption} that
\begin{align}\label{estimate10}
&\
|\mathbf{L}_t\phi(t,x)-\mathbf{L}_t\phi(t+\Delta,x)|\notag\\
\leq &\
C(|\partial_{xx}\phi(t,x)-\partial_{xx}\phi(t+\Delta,x)|+|\partial_{x}\phi(t,x)-\partial_{x}\phi(t+\Delta,x)|)\notag\\
&+
|H(t,x,\partial_{x}\phi(t,x))-H(t,x,\partial_{x}\phi(t+\Delta,x))|\notag\\
\leq &\ C\Delta(|\partial_{xxt}\phi|_{0}+|\partial_{xt}\phi|_{0}),
\end{align}
for some constant $C$ depending only on $[\phi]_{2,1}$ and $M$.
Combining estimates (\ref{estimate8})-(\ref{estimate10}), we easily
conclude.
\end{proof}\\

The following ``comparison-type" result for the approximation scheme
(\ref{semischeme}) will be used frequently in the next section. Most
of the arguments follow from Lemma 3.2 of \cite{BJ}, but we
highlight some key steps for the reader's convenience.

\begin{proposition}\label{schemecomparison}
Suppose that Assumption 2.1 is satisfied, and that $u$,
$v\in\mathcal{C}_b(\bar{Q}_T)$ are such that
\begin{align*}
S(\Delta,t,x,u,u(t+\Delta,\cdot)) \leq h_1\ \ \text{in}\ \bar{Q}_{T-\Delta};\\
S(\Delta,t,x,v,v(t+\Delta,\cdot))\ge h_2\ \ \text{in}\
\bar{Q}_{T-\Delta},
\end{align*}
for some $h_1$, $h_2\in\mathcal{C}_b(\bar{Q}_{T-\Delta})$. Then,
\begin{equation}\label{comparison_inequality}
u-v\leq \sup_{\bar{Q}_{T}\backslash
\bar{Q}_{T-\Delta}}(u-v)^{+}+(T-t)\sup_{\bar{Q}_{T-\Delta}}(h_{1}-h_{2})^{+}\
\ \text{in}\ \bar{Q}_{T}.
\end{equation}
\end{proposition}
\begin{proof}
We first note that without loss of generality, we may assume that
\begin{equation}\label{comparisonforscheme}
u\le v\ \ \text{in}\ \bar{Q}_{T}\backslash \bar{Q}_{T-\Delta}\ \
\text{and}\ \ h_{1}\leq h_{2}\ \ \text{in}\ \bar{Q}_{T-\Delta},
\end{equation}
since, otherwise, the function $w:=v+\sup_{\bar{Q}_{T}\backslash
\bar{Q}_{T-\Delta}}(u-v)^{+}+(T-t)\sup_{\bar{Q}_{T-\Delta}}(h_{1}-h_{2})^{+}$
satisfies $u\le w$ in $\bar{Q}_{T}\backslash \bar{Q}_{T-\Delta}$
and, by the monotonicity property (ii) in Proposition \ref{scheme
property},
\begin{align*}
S(\Delta,t,x,w,w(t+\Delta,\cdot)) &\ge
S(\Delta,t,x,v,v(t+\Delta,\cdot))+\sup_{\bar{Q}_{T-\Delta}}(h_{1}-h_{2})^{+}\\
&\ge h_{2}+\sup_{\bar{Q}_{T-\Delta}}(h_{1}-h_{2})^{+} \ge h_{1}\ \
\text{in}\ \bar{Q}_{T-\Delta}.
\end{align*}
Thus, it suffices to prove that $u\le v$ in $\bar{Q}_{T}$ when
(\ref{comparisonforscheme}) holds.

To this end, for $b\ge 0$, let $\psi_{b}(t):=b(T-t)$ and
$M(b):=\sup_{\bar{Q}_{T}}\{u-v-\psi_{b}\}.$ We need to show that
$M(0)\leq 0$. We argue by contradiction. If $M(0)>0$, then by the
continuity of $M$, we must have $M(b)>0$, for some $b>0$. For such
$b$, consider a sequence, say $\{(t_{n},x_{n})\}$ in $\bar{Q}_{T}$,
such that for $\delta(t,x):=M(b)-(u-v-\psi_{b})(t,x)$, we have
$\lim_{n\to\infty}\delta(t_n,x_n)=0.$ Since $M(b)>0$ but
$u-v-\psi_{b}\leq 0$ in $\bar{Q}_{T}\backslash \bar{Q}_{T-\Delta}$,
we must have $t_{n}\leq T-\Delta$ for sufficiently large $n$. Then,
for such $n$, we have
\begin{align*}
h_1(t_{n},x_{n}) \ge &\ S(\Delta,t_{n},x_{n},u(t_{n},x_{n}),u(t_{n}+\Delta,\cdot)) \\
 \ge &\
 S(\Delta,t_{n},x_{n},(v+\psi_{b}+M(b)-\delta)(t_{n},x_{n}),(v+\psi_{b}+M(b))(t_{n}+\Delta,\cdot)) \\
 \ge &\
 S(\Delta,t_{n},x_{n},v(t_{n},x_{n}),v(t_{n}+\Delta,\cdot))+(\psi_{b}(t_{n})-\psi_{b}(t_{n}+\Delta)-\delta(t_n,x_n))\Delta^{-1} \\
 \ge &\
 h_{2}(t_{n},x_{n})+b-\delta(t_n,x_n)\Delta^{-1}.
\end{align*}
On the other hand, since $h_{1}\leq h_{2}$ in $\bar{Q}_{T-\Delta}$,
we must have $b-\delta(t_n,x_n)\Delta^{-1}\leq 0$. Then, letting
$n\to\infty$, we deduce that $b\leq 0$, which is a contradiction.
\end{proof}\\

{Following along similar argument, we also obtain the comparison
inequality (\ref{comparison_inequality}) on the partition grid
$\bar{\mathcal{G}}^\Delta_{T}:
\{0<\Delta<2\Delta<\dots<T-\Delta<T\}$.}

{\begin{corollary}\label{schemecompco}
%Consider the grid $\bar{\mathcal{G}}^\Delta_{T}\subset\bar{Q}_{T}$, and
Let $\mathcal{G}^\Delta_{T}:=\bar{\mathcal{G}}^\Delta_{T}\backslash
\{T\}$ be the partition grid before terminal time $T$. Suppose that
$u$, $v\in\mathcal{C}_b(\bar{Q}_T)$ are such that
\begin{align*}
S(\Delta,t,x,u,u(t+\Delta,\cdot)) \leq h_1\ \ \text{in}\ \mathcal{G}^\Delta_{T};\\
S(\Delta,t,x,v,v(t+\Delta,\cdot))\ge h_2\ \ \text{in}\
\mathcal{G}^\Delta_{T},
\end{align*}
for some $h_1$, $h_2\in\mathcal{C}_b(\mathcal{G}^\Delta_{T})$. Then,
\begin{equation}
u-v\leq |(u(T,\cdot)-v(T,\cdot))^{+}|_0+(T-t)|(h_{1}-h_{2})^{+}|_0\
\ \text{in}\ \bar{\mathcal{G}}^\Delta_{T}.
\end{equation}
\end{corollary}}

%%%%%%%%%%%%%%%%%%%%%%%%%%%%%%%%%%%%%%%%%
%%%%%%%%%%%%%%%%%%%%%%%%%%%%%%%%%%%%%%%%%%%%%

%%%%%%%%%%%%%%%%%%%%%%%%%%%%%%%%%%%%%%%%%%%%%%%%%%%%%%%%%%%%%%%%%%%%%%
%%%%%%%%%%%%%%%%%%%%%%%%%%%%%%%%%%%%%%%%%%%%%%%%%%%%%%%%%%%%%%%%%
\section{Convergence rate of the approximation scheme}

{The classical convergence theory of Barles-Souganidis (see
\cite{Barles}) will only imply the convergence of the approximate
solution $u^{\Delta}$ to the viscosity solution $u$ of equation
(\ref{PDE_1}). To further determine the convergence rate of
$u^{\Delta}$ to $u$, we establish upper and lower bounds on the
approximation error.}

We start with the special case when (\ref{PDE_1}) has a unique
smooth solution $u$ with bounded derivatives of any order.}

\begin{theorem}\label{smoothcase} Suppose that Assumption \ref{data assumption} is
satisfied and that equation (\ref{PDE_1}) admits a unique smooth
solution $u\in\mathcal{C}_b^{\infty}(\bar{Q}_T)$. Then, there exists
a constant $C$, depending only on $M$ and $T$, such that
$$|u-u^{\Delta}|\leq C\Delta\ \ \text{in}\ \bar{Q}_T.$$
\end{theorem}

\begin{proof} Using that $u\in\mathcal{C}_b^{\infty}(\bar{Q}_T)$,
the consistency error estimate (\ref{consistant_error}) yields
\begin{align*}
&|S(\Delta,t,x,u(t,x),u(t+\Delta,\cdot))|\\
\leq&\
C\Delta\left(|\partial_{tt}u|_{0}+|\partial_{xxxx}u|_0+|\partial_{xxt}u|_{0}+\mathcal{R}(u)\right)
\leq C\Delta,
\end{align*}
 for
$(t,x)\in\bar{Q}_{T-\Delta}$. On the other hand, from the definition
of the approximation scheme (\ref{semischeme}), we have
$$S(\Delta,t,x,u^{\Delta}(t,x),u^{\Delta}(t+\Delta,\cdot))=0,$$
for $(t,x)\in\bar{Q}_{T-\Delta}$. In turn, the comparison result in
Proposition \ref{schemecomparison} yields
$$u(t,x)-u^{\Delta}(t,x)\leq \sup_{(t,x)\in\bar{Q}_T\backslash\bar{Q}_{T-\Delta}}(u(t,x)-u^{\Delta}(t,x))^++(T-t)C\Delta,$$
and
$$u^{\Delta}(t,x)-u(t,x)\leq \sup_{(t,x)\in\bar{Q}_T\backslash\bar{Q}_{T-\Delta}}(u^{\Delta}(t,x)-u(t,x))^++(T-t)C\Delta,$$
for $(t,x)\in\bar{Q}_T$. It is left to prove that
$|u-u^\Delta|<C\Delta$ in $\bar{Q}_T\backslash\bar{Q}_{T-\Delta}$.
Indeed, the comparison result in Corollary \ref{schemecompco} yields
$$|u(T-\Delta,x)-u^\Delta(T-\Delta,x)|\le (T-(T-\Delta))C\Delta = C\Delta^2,$$
and thus, in $\bar{Q}_T\backslash\bar{Q}_{T-\Delta}$,
\begin{align*}
&\ |u(t,x)-u^\Delta(t,x)| = |u(t,x)-\omega_1(t)U(x)-\omega_2(t)u^{\Delta}(T-\Delta,x)| \\
=&\ |u(t,x)-U(x)+\omega_2(t)(U(x)-u^{\Delta}(T-\Delta,x))|\\
\le &\
|u(t,x)-U(x)|+\omega_2(t)\left(|U(x)-u(T-\Delta,x)|+|u(T-\Delta,x)-u^\Delta(T-\Delta,x)|\right)\\
\le &\ [u]_{1,1}\Delta+\omega_2(t)  [u]_{1,1}\Delta +
\omega_2(t)C\Delta^2 \le C\Delta.
\end{align*} We easily conclude.
\end{proof}\\

In general, the above result might not hold as (\ref{PDE_1}) only
admits a viscosity solution $u\in\mathcal{C}_b^1(\bar{Q}_T)$ due to
possible degeneracies. A natural idea is then to approximate the
viscosity solution $u$ by a sequence of smooth sub- and
supersolutions $u_{\varepsilon}$ and, in turn, compare them with
$u^{\Delta}$ using the comparison result for the approximation
scheme developed in Proposition \ref{schemecomparison}. We carry out
this {procedure} next.

%%%%%%%%%%%%%%%%%%%%%%%%%%%%%%%%%%%%%%%%%%%%%%%%%%%

%%%%%%%%%%%%%%%%%%%%%%%%%%%%%
\subsection{Upper bound for the approximation error}

We derive an upper bound for the approximation error $u-u^{\Delta}$.
We do so by first constructing a sequence of smooth subsolutions to
equation (\ref{PDE_1}) by perturbing its coefficients. As we
mentioned in the introduction, this approach, known as the
\emph{shaking coefficients technique}, was initially proposed by
Krylov \cite{Krylov1} \cite{Krylov}, and further developed by Barles
and Jakobsen \cite{BJ1} \cite{Jakobsen}.

To this end, for $\varepsilon\in[0,1]$, we extend the functions
$f:=\sigma, b$ and $H$ to
${Q}^{-\varepsilon^{2}}_{T+\varepsilon^{2}}:=[-\varepsilon^{2},T+\varepsilon^{2})\times\mathbb{R}^{n}$
and
${Q}^{-\varepsilon^{2}}_{T+\varepsilon^{2}}\times\mathbb{R}^{n}$,
respectively, so that Assumption \ref{data assumption} still holds.
We then define $f^\theta(t,x):=f(t+\tau,x+e)$ and
$H^{\theta}(t,x,p):=H(t+\tau,x+e,p)$, where $\theta=(\tau,e)$ with
$\theta\in\Theta^\varepsilon:=[-\varepsilon^2,0]\times\varepsilon
B(0,1)$, and consider the perturbed version of equation
(\ref{PDE_1}), namely,
\begin{equation}
\label{PtbedPDE} \left\{\begin{array}{r}
\displaystyle-\partial_tu^{\varepsilon}+\sup_{\theta\in\Theta^{\varepsilon}}\left\{-\frac{1}{2}\text{Trace}\left(\sigma^\theta{\sigma^\theta}^T(t,x)\partial_{xx}u^{\varepsilon}\right)-b^\theta(t,x)\cdot\partial_xu^{\varepsilon}
+H^{\theta}(t,x,\partial_xu^{\varepsilon})\right\}=0\\
\text{in}\ Q_{T+\varepsilon^2};\\
\displaystyle u^{\varepsilon}(T+\varepsilon^{2},x)=U(x)\ \text{in}\
\mathbb{R}^n.
\end{array}\right.
\end{equation}
%\begin{equation}\label{PtbedPDE}
%-\partial_tu^{\varepsilon}(t,x)+\sup_{\theta\in\Theta^{\varepsilon}}\left\{-\frac{1}{2}\text{Trace}\left(\sigma^\theta{\sigma^\theta}^T(t,x)\partial_{xx}u^{\varepsilon}(t,x)\right)-b^\theta(t,x)\cdot\partial_xu^{\varepsilon}(t,x)+H^{\theta}(t,x,\partial_xu^{\varepsilon}(t,x))\right\}=0,
%\end{equation}
%with terminal condition
%\begin{equation}\label{Ptbed_terminal}
%u^{\varepsilon}(T+\varepsilon^{2},x)=U(x),
%\end{equation}
 Note that when the
perturbation parameter $\varepsilon=0$, equations (\ref{PtbedPDE})
and (\ref{PDE_1}) coincide.

We establish existence, uniqueness and regularity results for the
HJB equation (\ref{PtbedPDE}), and a comparison between $u$ and
$u^{\varepsilon}$. Their proofs are provided in Appendix \ref{App1}.

\begin{proposition}\label{upperbound_1}
Suppose that Assumption \ref{data assumption} is satisfied. Then,
there exists a unique viscosity solution
$u^{\varepsilon}\in\mathcal{C}^{1}_b(\bar{Q}_{T+\varepsilon^2})$ of
the HJB equation (\ref{PtbedPDE}), with $|u^{\varepsilon}|_{1}\leq
C$, for some constant $C$ depending only on $M$ and $T$. Moreover,
\begin{equation}\label{est2}
 {|u-u^{\varepsilon}|}\leq C\varepsilon\ \ \text{in}\ \bar{Q}_T.
\end{equation}
\end{proposition}

Next,  we regularize $u^{\varepsilon}$ by a standard mollification
procedure. For this, let $\rho(t,x)$ be a $\mathbb{R}_+$-valued
smooth function with compact support $\{-1<t<0\}\times\{|x|< 1\}$
and mass $1$, and introduce the sequence of mollifiers
$\rho_{\varepsilon}$,
\begin{equation}\label{mollifer}
\rho_{\varepsilon}(t,x):=\frac{1}{\varepsilon^{n+2}}\rho\left(\frac{t}{\varepsilon^2},\frac{x}{\varepsilon}\right).
\end{equation}
For $(t,x)\in \bar{Q}_T$, we then define
$$u_{\varepsilon}(t,x)=u^{\varepsilon}*
\rho_{\varepsilon}(t,x)=\int_{-\varepsilon^2< \tau< 0}\int_{|e|<
\varepsilon}u^{\varepsilon}(t-\tau,x-e)\rho_{\varepsilon}(\tau,e)ded\tau.$$
Standard properties of mollifiers imply that
$u_{\varepsilon}\in\mathcal{C}_b^{\infty}(\bar{Q}_{T})$,
\begin{equation}\label{upperbound_2}
|u^{\varepsilon}-u_{\varepsilon}|_0\leq C\varepsilon,
\end{equation}
and, moreover, for positive integers $i$ and $j$,
\begin{equation}\label{mollifier}
|\partial_{t}^i\partial_{x}^ju_{\varepsilon}|_0\leq
C\varepsilon^{1-2i-|j|},
\end{equation}
where the constant $C$ is independent of $\varepsilon$.

We observe that the function $u^{\varepsilon}(t-\tau,x-e)$,
$(t,x)\in Q_T$, is a viscosity subsolution of equation (\ref{PDE_1})
in $Q_T$, for any {$(\tau,e)\in\Theta^{\varepsilon}$}. On the other
hand, a Riemann sum approximation shows that $u_{\varepsilon}(t,x)$
can be viewed as the limit of convex combinations of
$u^{\varepsilon}(t-\tau,x-e)$, for
{$(\tau,e)\in\Theta^{\varepsilon}$}. Since the equation in
(\ref{PDE_1}) is convex in $\partial_{x}u$, and linear in
$\partial_tu$ and $\partial_{xx}u$, the convex combinations of
$u^{\varepsilon}(t-\tau,x-e)$ are also subsolutions of (\ref{PDE_1})
in $Q_T$. Using the stability of viscosity solutions, we deduce that
$u_{\varepsilon}(t,x)$ is also a subsolution of (\ref{PDE_1}) in
$Q_T$.

%\begin{proof}
%We split the consistent error
%$\Delta\mathcal{E}(t,\Delta,u_{\varepsilon})$ into two parts as follows:
%\begin{align*}
%&|u_{\varepsilon}(t-\Delta,x)-\mathbf{S}_{t-\Delta}(\Delta)u_{\varepsilon}(t,x)+\Delta\partial_tu_{\varepsilon}(t,x)-\Delta\mathbf{L}_t u_{\varepsilon}(t,x)|_0\\
%\leq &\
%|u_{\varepsilon}(t,x)-\mathbf{S}_{t-\Delta}(\Delta)u_{\varepsilon}(t,x)-\Delta\mathbf{L}_t u_{\varepsilon}(t,x)|_0
%+
%|u_{\varepsilon}(t,x)-u_{\varepsilon}(t-\Delta,x)-\Delta\partial_tu_{\varepsilon}(t,x)|_0.
%\end{align*}
%
%The estimate of the first term has been shown in (vi) of Lemma
%\ref{semigroup}. Furthermore, by applying the estimates of the
%mollifiers (\ref{mollifier}) and keeping the worst terms involving
%$\varepsilon$, we obtain that
%\begin{equation}\label{estimate5}
%|u_{\varepsilon}(t,x)-\mathbf{S}_{t-\Delta}(\Delta)u_{\varepsilon}(t,x)-\Delta\mathbf{L}_t u_{\varepsilon}(t,x)|_0\leq C\Delta^2\varepsilon^{-3}.
%\end{equation}
%
%For the second term, using Taylor expansion, we get
%\begin{align}\label{estimate9}
%&\
%|u_{\varepsilon}(t,x)-u_{\varepsilon}(t-\Delta,x)-\Delta\partial_tu_{\varepsilon}(t,x)|_0\notag\\
%\leq&\
%|\int_{t-\Delta}^{t}\left(\partial_{t}u_{\varepsilon}(t,x)-\int_v^{t}\partial_{tt}u_{\varepsilon}(u,x)du\right)dv-
%\Delta\partial_tu_{\varepsilon}(t,x)|_0\notag\\
%\leq&\ C\Delta^2|\partial_{tt}u_{\varepsilon}|_0\leq
%C\Delta^2\varepsilon^{-3}.
%\end{align}
%
%Hence, the consistent error estimate (\ref{estimate4}) follows from
%(\ref{estimate5}) and (\ref{estimate9}).
%\end{proof}

We are now ready to establish an upper bound for the approximation
{error}.

\begin{theorem}\label{theorem_error_1}
Suppose that Assumption \ref{data assumption} holds. Let
$u^{\Delta}\in\mathcal{C}_b(\bar{Q}_{T})$ satisfy the approximation
scheme (\ref{semischeme}) and $u\in\mathcal{C}_b^1(\bar{Q}_{T})$ be
the unique viscosity solution of equation (\ref{PDE_1}). Then, there
exists a constant $C$, depending only on $M$ and $T$, such that
$$u-u^{\Delta}\leq C\Delta^{\frac14}\ \ \text{in}\ \bar{Q}_{T}.$$
\end{theorem}

\begin{proof}
Substituting $u_{\varepsilon}$ into the consistency error estimate
(\ref{consistant_error}) and using (\ref{mollifier}) give
\begin{align*}\label{estimate4}
%\mathcal{E}(t,\Delta,u_{\varepsilon}):=&\
&\left|-\partial_tu_{\varepsilon}(t,x)+{\mathbf{L}_t}u_{\varepsilon}(t,x)-S(\Delta,t,x,u_{\varepsilon}(t,x),u_{\varepsilon}(t+\Delta,\cdot))\right|\notag\\
\leq&\
C\Delta\left(|\partial_{tt}u_{\varepsilon}|_{0}+|\partial_{xxxx}u_{\varepsilon}|_0+|\partial_{xxt}u_{\varepsilon}|_{0}+\mathcal{R}(u_{\varepsilon})\right)
\leq C\Delta\varepsilon^{-3},
\end{align*}
for $(t,x)\in\bar{Q}_{T-\Delta}$. Since $u_{\varepsilon}$ is a
subsolution of (\ref{PDE_1}) in $Q_T$, we  have
\begin{equation*}
S(\Delta,t,x,u_{\varepsilon}(t,x),u_{\varepsilon}(t+\Delta,\cdot))\leq
C\Delta\varepsilon^{-3},
\end{equation*}
for $(t,x)\in\bar{Q}_{T-\Delta}$. Furthermore, by the definition of
the approximation scheme (\ref{semischeme}), we also have
$$S(\Delta,t,x,u^{\Delta}(t,x),u^{\Delta}(t+\Delta,\cdot))=0,$$
for $(t,x)\in\bar{Q}_{T-\Delta}$. In turn, Proposition
\ref{schemecomparison} implies
$$u_{\varepsilon}-u^{\Delta}\leq \sup_{\bar{Q}_{T}\backslash \bar{Q}_{T-\Delta}}(u_{\varepsilon}-u^{\Delta})^{+}+C(T-t)\Delta\varepsilon^{-3}\ \ \text{in}\ \bar{Q}_T.$$
Next, using estimates (\ref{est2}) and (\ref{upperbound_2}), we
obtain that $|u-u_{\varepsilon}|\leq C\varepsilon$ and, thus,
\begin{align*}
u-u^{\Delta}&=(u-u_{\varepsilon})+(u_{\varepsilon}-u^{\Delta})\\
&\leq C\varepsilon+\sup_{\bar{Q}_{T}\backslash
\bar{Q}_{T-\Delta}}(u_{\varepsilon}-u^{\Delta})^{+}+C(T-t)\Delta\varepsilon^{-3}\\
&\leq \sup_{\bar{Q}_{T}\backslash
\bar{Q}_{T-\Delta}}(u-u^{\Delta})^{+}+C(\varepsilon+\Delta\varepsilon^{-3})\
\ \text{in}\ \bar{Q}_T.
\end{align*}
By choosing $\varepsilon={\Delta}^{\frac14}$, we further deduce that
$$u-u^{\Delta}\leq \sup_{\bar{Q}_{T}\backslash \bar{Q}_{T-\Delta}}(u-u^{\Delta})^{+}+C\Delta^{\frac14}\ \ \text{in}\ \bar{Q}_T.$$
We conclude using estimate (\ref{estimate_final_interval}) in Lemma
\ref{errorsmall}.
\end{proof}

%%%%%%%%%%%%%%%%%%%%%%%%%%%%%%%%%%%%%%%%%%%%%%%%%%%%%%%%%%%%%%%%%%%%%%%%%%
\subsection{Lower bound for the approximation {error}}

{To obtain a lower bound of $u-u^{\Delta}$, we cannot follow the
above perturbation procedure to construct approximate smooth
supersolutions to equation (\ref{PDE_1}). This is because if we
perturb its coefficients to obtain a viscosity supersolution, its
convolution with the mollifier may no longer be a supersolution due
to the convexity of equation (\ref{PDE_1}) with respect to its
terms. Furthermore, interchanging the roles (as in \cite{HL1}) of
equation (\ref{PDE_1}) and its approximation scheme
(\ref{semischeme}) does not work either, because the solution
$u^{\Delta}$ of the approximation scheme (and its perturbation
solution) may in general lose the H\"older and Lipschitz continuity
in $(t,x)$. This is due to the lack of the continuous dependence
result for the approximation scheme, compared with the continuous
dependence result for equation (\ref{PDE_1}) and its perturbation
equation (\ref{PtbedPDE}) (see Lemma \ref{Ptbedproperty}).}

To overcome these difficulties, we follow the idea of Barles and
Jakobsen \cite{BJ} to build approximate supersolutions which are
smooth at the ``right points'' by introducing an appropriate optimal
switching stochastic control system. To apply this method to the
problem herein, we first observe that, using the convex dual
function $L$ introduced in (\ref{L}), we can write equation
(\ref{PDE_1}) as a HJB equation, namely,
\begin{equation}
\label{HJBeq} \left\{\begin{array}{ll}
\displaystyle -\partial_{t}u+\sup_{q\in\mathbb{R}^{n}}\mathcal{L}^{q}\left(t,x,\partial_{x}u,\partial_{xx}u\right)=0&\text{in}\  Q_T;\\
\displaystyle u(T,x)=U(x)&\text{in}\ \mathbb{R}^n,
\end{array}\right.
\end{equation}
%\begin{equation}\label{HJBeq}
%-\partial_{t}u(t,x)+\sup_{q\in\mathbb{R}^{n}}\mathcal{L}^{q}\left(t,x,\partial_{x}u(t,x),\partial_{xx}u(t,x)\right)=0,
%\end{equation}
with
\begin{equation*}
\mathcal{L}^{q}(t,x,p,X):=-\frac{1}{2}\text{Trace}\left(\sigma\sigma^T(t,x)X\right)-(b(t,x)-q)\cdot
p-L(t,x,q).
\end{equation*}
It then follows from Proposition \ref{transformlemma} (iv) that the
supremum can be achieved at some point, say $q^{*}$, with
$|q^{*}|\leq\xi(|\partial_{x}u|)$. Furthermore, Proposition
\ref{solutionproperty} implies that $|q^{*}|\leq C$, for some
constant $C$ depending only on $M$ and $T$.
%This in turn suggests that $q^{*}$ is bounded uniformly by some constant $C$ depending only on $T$, $M$, and the Legendre transform $L$.
Thus, we rewrite the equation  in (\ref{HJBeq}) as
\begin{equation*}
-\partial_{t}u+\sup_{q\in
K}\mathcal{L}^{q}\left(t,x,\partial_{x}u,\partial_{xx}u\right)=0,
\end{equation*}
where $K\subset\mathbb{R}^{n}$ is a compact set. Since  $K$ is
separable, it has a countable dense subset, say
$K_{\infty}=\{q_{1},q_{2},q_{3},...\}$ and, in turn, the continuity
of $\mathcal{L}^{q}$ in $q$ implies that
\begin{equation*}
\sup_{q\in K}\mathcal{L}^{q}(t,x,p,X)=\sup_{q\in
K_{\infty}}\mathcal{L}^{q}(t,x,p,X).
\end{equation*}
Therefore, the equation in (\ref{HJBeq}) further reduces to
\begin{equation*}
-\partial_{t}u+\sup_{q\in
K_{\infty}}\mathcal{L}^{q}\left(t,x,\partial_{x}u,\partial_{xx}u\right)=0.
\end{equation*}

For $m\geq 1$, we now consider the approximations of (\ref{HJBeq}),
\begin{equation}
\label{finiteHJB} \left\{\begin{array}{ll}
\displaystyle-\partial_{t}u^{m}+\sup_{q\in
{K}_{m}}\mathcal{L}^{q}\left(t,x,\partial_{x}u^{m},\partial_{xx}u^{m}\right)=0&\text{in}\  Q_T;\\
\displaystyle u^{m}(T,x)=U(x)&\text{in}\ \mathbb{R}^n,
\end{array}\right.
\end{equation}
%\begin{equation}\label{finiteHJB}
%-\partial_{t}u^{m}(t,x)+\sup_{q\in
%{K}_{m}}\mathcal{L}^{q}\left(t,x,\partial_{x}u^{m}(t,x),\partial_{xx}u^{m}(t,x)\right)=0,
%\end{equation}
%with terminal condition
%\begin{equation}\label{finiteHJB_terminal}
%u^{m}(T,x)=U(x),
%\end{equation}
where ${K}_{m}:=\{q_{1},...,q_{m}\}\subset K_{\infty}$, i.e. $K_m$
consists of the first $m$ points in $K_{\infty}$ and satisfies
$\cup_{m\geq 1}K_m=K_{\infty}$. It then follows from Proposition 2.1
of \cite{BJ} that (\ref{finiteHJB}) admits a unique viscosity
solution $u^m\in\mathcal{C}_b^1(\bar{Q}_T)$, with $|u^m|_{1}\leq C$,
for some constant $C$ depending only on $M$ and $T$. Furthermore,
Arzela-Ascoli's theorem yields that there exists a subsequence of
$\{u^{m}\}$, still denoted as $\{u^{m}\}$, such that, as
$m\rightarrow\infty$,
\begin{equation}\label{lemma1}
u^m(t,x)\rightarrow u(t,x)\ \ \text{uniformly\ in}\
(t,x)\in\bar{Q}_{T}.
\end{equation}

%\begin{lemma}\label{lemma1}
%Suppose that Assumption \ref{data assumption} is satisfied. Then,
%there exists a unique viscosity solution $u^m$, with
%$u^m\in\mathcal{C}_b^1(\bar{Q}_T)$, of the HJB equation
%(\ref{switching})-(\ref{switching_terminal}), with $|u^m|_{1}\leq C$ depending only on $M$ and
%$T$. Moreover, $u^m\rightarrow u$ uniformly in $(t,x)\in\bar{Q}_{T}$
%as $m\rightarrow\infty$.
%\end{lemma}
%
%\begin{proof}
%We first show that we can indeed use $u^{m}$ in (\ref{finiteHJB}) to
%approximate $u$. This is true since by standard regularity results
%for HJB equations (see Proposition 2.1 of \cite{BJ}), we have
%$$|u^{m}|_{1}\leq C$$
%where the constant $C$ depends only on $M$ and $T$. Then by
%Arzela-Ascoli's theorem, there exists a subsequence of
%$\{u^{m}\}_{m}$, also denoted by $\{u^{m}\}_{m}$, converges locally
%uniformly. Obviously, this limit function satisfies
%$$
%-\partial_{t}u(t,x)+\sup_{q\in
%K_{\infty}}\mathcal{L}^{q}(t,x,\partial_{x}u,\partial_{xx}u)(t,x)=0,
%$$
%\end{proof}

Next, we construct a sequence of (local) smooth supersolutions to
approximate $u^m$. For this, we consider the optimal switching
system
\begin{equation}
\label{switching} \left\{\begin{array}{ll}
\displaystyle\max\left\{-\partial_tv_i+\mathcal{L}^{q_{i}}(t,x,\partial_{x}v_i,\partial_{xx}v_i),v_i-\mathcal{M}^{k}_{i}v
\right\}=0&\text{in}\  Q_T;\\
\displaystyle v_i(T,x)=U(x)&\text{in}\ \mathbb{R}^n,
\end{array}\right.
\end{equation}
%\begin{equation}\label{switching}
%\max\left\{-\partial_tv_i+\mathcal{L}^{q_{i}}(t,x,\partial_{x}v_i,\partial_{xx}v_i),v_i-\mathcal{M}^{k}_{i}v
%\right\}=0,
%\end{equation}
%with terminal condition
%\begin{equation}\label{switching_terminal}
%v_i(T,x)=U(x),
%\end{equation}
where $i\in\mathcal{I}:=\{1,...,m\}$ and
$\mathcal{M}^{k}_{i}v:=\min_{j\neq i,\ j\in\mathcal{I}}\{v_{j}+k\},$
for some constant $k>0$ representing the switching cost.

%We establish the following well-posedness and regularity results for
%the optimal switching system (\ref{switching}), and as well as a
%comparison result for its solution $v_i$ and the approximation
%$u^m$.

\begin{proposition}\label{proposition}
Suppose that Assumption \ref{data assumption} is satisfied. Then,
there exists a unique viscosity solution $v=(v_1,\dots,v_m)$ of the
optimal switching system (\ref{switching}) such that $|v|_{1}\leq
C$, for some constant $C$ depending only on $M$ and $T$. Moreover,
for $i\in\mathcal{I}$,
\begin{equation}\label{optimal_switching_approx}
0\leq v_{i}-u^m\leq C(k^{\frac13}+k^{\frac23})\ \ \text{in}\
\bar{Q}_{T}.
\end{equation}
\end{proposition}

%In fact, as $k\to0$, every component of $v$ converges locally
%uniformly to the solution of the following HJB equation:
%\begin{equation}\label{HJBeq2}
%-\partial_{t}u+\sup_{q\in
%\mathcal{A}}\mathcal{L}^{q}(t,x,\partial_{x}u,\partial_{xx}u)=0\
%\text{in}\ Q_{T}
%\end{equation}
%with terminal condition $u(T,x)=U(x)$, where
%$\mathcal{A}=\cup_{i=1}^{m}\mathcal{A}_{i}$. Note that, by Remark
%\ref{transform}, if we set $\mathcal{A}=K$, the solution $u$ is
%exactly the solution of our original PDE (\ref{PDE_1}). More
%specificly, we have the following error estimate result:
%\begin{theorem}\label{errorboundswitching}
%Under Assumption \ref{data assumption}, if
%$\mathcal{A}=\cup_{i=1}^{m}\mathcal{A}_{i}\subset\mathbb{R}^{n}$ is
%a compact set, and let $u$ and $v$ be the solutions of
%(\ref{HJBeq2}) with (\ref{terminal}) and (\ref{switching})-(\ref{switching_terminal})
%respectively, then for any $k>0$, there exists a constant $C$ only
%depending only $T$, $M$ and $\mathcal{A}$ such that
%$$0\leq v_{i}-u\leq C(k^{\frac13}+k^{\frac23}), \ \ \text{in}\ Q_{T}, \ \ \ i\in\mathcal{I}.$$
%\end{theorem}

The proof essentially follows from Proposition 2.1 and Theorem 2.3
of \cite{BJ} and it is thus omitted. We only remark that since we do
not require the switching cost to satisfy $k\leq 1$, we keep the
term $k^{\frac23}$ in the above estimate. This will not affect the
convergence rate of the approximation scheme.

Next, still following the approach of \cite{BJ}, we construct smooth
approximations of $v_i$. Since in the continuation region of
(\ref{switching}), the solution $v_i$ satisfies the \emph{linear}
equation, namely,
$$-\partial_tv_i+\mathcal{L}^{q_{i}}(t,x,\partial_{x}v_i,\partial_{xx}v_i)=0\ \ \text{in}\ \{(t,x)\in{Q}_T:v_i(t,x)<\mathcal{M}_i^{k}v(t,x)\},$$
we may perturb its coefficients to obtain a sequence of smooth
supersolutions. This will in turn give a lower bound of the error
$u^m-u^{\Delta}$. A subtle point herein is how to identify the
continuation region by appropriately choosing the switching cost
$k$. For this, we follow the idea used in Lemma 3.4 of \cite{BJ}.

\begin{proposition}\label{lowbd1}
Suppose that Assumption \ref{data assumption} holds. Let
$u^{\Delta}\in\mathcal{C}_b(\bar{Q}_{T})$ satisfy the approximation
scheme (\ref{semischeme}) and $u^m\in\mathcal{C}_b^1(\bar{Q}_{T})$
be the unique viscosity solution of the HJB equation
(\ref{finiteHJB}). Then, there exists a constant $C$, depending only
on $M$ and $T$, such that
$$u^{\Delta}-u^{m}\leq \sup_{\bar{Q}_{T}\backslash\bar{Q}_{T-\Delta}}(u^{\Delta}-u^{m})^{+}+C\Delta^{\frac{1}{10}}\ \ \text{in}\ \bar{Q}_{T}.$$
\end{proposition}

\begin{proof}
Let $\varepsilon\in[0,1]$. In analogy to (\ref{PtbedPDE}), we
perturb the coefficients of the optimal switching system
(\ref{switching}) and consider
\begin{equation}
\label{Ptbedswitching1} \left\{\begin{array}{ll}
\max\displaystyle\left\{-\partial_tv_i^{\varepsilon}+\inf_{(\tau,e)\in\Theta^{\varepsilon}}\mathcal{L}^{q_i}(t+\tau,x+e,\partial_xv_i^{\varepsilon},\partial_{xx}v_i^{\varepsilon}),
v_i^{\varepsilon}-\mathcal{M}^{k}_{i}v^{\varepsilon} \right\}=0
&\text{in}\ Q_{T+\varepsilon^2};\\
\displaystyle v_i^{\varepsilon}(T+\varepsilon^{2},x)=U(x)&\text{in}\
\mathbb{R}^n.
\end{array}\right.
\end{equation}
%\begin{equation}\label{Ptbedswitching1}
%\max\left\{-\partial_tv_i^{\varepsilon}+\inf_{(\tau,e)\in\Theta^{\varepsilon}}\mathcal{L}^{q_i}(t+\tau,x+e,\partial_xv_i^{\varepsilon},\partial_{xx}v_i^{\varepsilon}),
%v_i^{\varepsilon}-\mathcal{M}^{k}_{i}v^{\varepsilon} \right\}=0,
%\end{equation}
%with terminal condition
%\begin{equation}\label{Ptbedswitching1_terminal}
%v_i^{\varepsilon}(T+\varepsilon^{2},x)=U(x),
%\end{equation}
%where $\Theta^{\varepsilon}=[-\varepsilon^{2},0]\times \varepsilon
%B(0,1)$, and the coefficients $\sigma$, $b$ and $L$ are
%appropriately extended such that Assumption \ref{data assumption}
%still holds.
It then follows from Proposition 2.2 of \cite{BJ} that
(\ref{Ptbedswitching1}) admits a unique viscosity solution, say
$v^{\varepsilon}=(v_1^{\varepsilon},\dots,v_m^{\varepsilon})$, with
$|v^{\varepsilon}|_{1}\leq C$ and, moreover, for each
$i\in\mathcal{I}$,
\begin{equation}\label{optimal_switching_approx_1}
{|v^{\varepsilon}_i-v_i|}\leq C\varepsilon\ \ \text{in}\ \bar{Q}_T,
\end{equation} where the constant $C$ depends only on
$M$ and $T$.
%the existence and uniqueness of the (\ref{Ptbedswitching1}-\ref{Ptbedswitching1_terminal})
%\begin{lemma}\label{Ptbedswitchingproperty1}
%Under Assumption \ref{data assumption}, if $\mathcal{A}=\cup_{i=1}^{m}\mathcal{A}_{i}\subset\mathbb{R}^{n}$
%is a compact set, then there exists a unique solution $v^{\varepsilon}$ of (\ref{Ptbedswitching1}-\ref{Ptbedswitching1_terminal}) satisfying $|v^{\varepsilon}|_{1}\leq C$ and $|v^{\varepsilon}-v|_{0}\leq C\varepsilon$, where $v$ solves (\ref{switching})-(\ref{switching_terminal}) and the constant $C$ depends only on $T$, $M$ and $\mathcal{A}$.
%\end{lemma}
%Now set $\mathcal{A}_{i}=\{q_{i}\}$ for $i\in\mathcal{I}$ and hence $\cup_{i=1}^{m}\mathcal{A}_{i}=\mathcal{A}^{m}$ is compact. Note that under this setting, in the definition of $F_{i}^{k,\varepsilon}$ we can move out the ``sup'' operator and replace $\mathcal{L}^{q}$ by $\mathcal{L}^{q_{i}}$ since $q$ must be $q_{i}$. We then apply Theorem \ref{errorboundswitching} to get
%$$0\leq v_{i}-u^{m}\leq C(k^{\frac13}+k^{\frac23})\ \text{in} \ Q_{T}, \ \ \ i\in\mathcal{I},$$
%where $v$ is the solution of (\ref{switching})-(\ref{switching_terminal}) under our setting. Together with Lemma \ref{Ptbedswitchingproperty1} we have
In turn, inequalities (\ref{optimal_switching_approx}) and
(\ref{optimal_switching_approx_1}) imply that, for each
$i\in\mathcal{I}$,
\begin{equation}\label{Ptbederror}
{|v_{i}^{\varepsilon}- u^{m}|}\leq
{|v_{i}^{\varepsilon}-v_i|}+{|v_{i}-u^m|}\leq
C(\varepsilon+k^{\frac13}+k^{\frac23})\ \ \text{in}\ \bar{Q}_T.
\end{equation}

Next, we regularize $v_i^{\varepsilon}$ by introducing
$v_{i,\varepsilon}(t,x):=v_{i}^{\varepsilon}*\rho_{\varepsilon}(t,x)$,
for $(t,x)\in\bar{Q}_{T}$, where $\rho_{\varepsilon}$ is the
mollifer defined in (\ref{mollifer}). Then,
$v_{i,\varepsilon}\in\mathcal{C}_{b}^{\infty}(\bar{Q}_{T})$,
\begin{equation}\label{convoerror}
|v_{i,\varepsilon}-v_{i}^{\varepsilon}|_{0}\leq C\varepsilon,
\end{equation}
and, moreover, for positive integers $m$ and $n$,
\begin{equation}\label{convoproperty}
{|\partial_{t}^{m}\partial_{x}^{n}v_{i,\varepsilon}|_0\leq
C\varepsilon^{1-2m-|n|}}.
\end{equation}

We introduce the function
$w_{\varepsilon}:=\min_{i\in\mathcal{I}}v_{i,\varepsilon},$ which is
smooth in $\bar{Q}_T$ except for finitely many points. Then,
(\ref{Ptbederror}) and (\ref{convoerror}) yield
\begin{equation}\label{errtogether}
{|u^{m}-w_{\varepsilon}|}\leq
C(\varepsilon+k^{\frac13}+k^{\frac23})\ \ \text{in}\ \bar{Q}_T.
\end{equation}
For each $(t,x)\in\bar{Q}_{T}$, let
$j:=\arg\min_{i\in\mathcal{I}}v_{i,\varepsilon}(t,x)$. Then,
$w_{\varepsilon}(t,x)=v_{j,\varepsilon}(t,x)$ and, for such $j$, we
obtain that
$$v_{j,\varepsilon}(t,x)-\mathcal{M}_{j}^{k}v_{\varepsilon}(t,x)=\max_{i\neq j, i\in\mathcal{I}}\{v_{j,\varepsilon}(t,x)-v_{i,\varepsilon}(t,x)-k\}\leq
-k.$$ In turn, inequality (\ref{convoerror}) implies that
$$v_{j}^{\varepsilon}(t,x)-\mathcal{M}_{j}^{k}v^{\varepsilon}(t,x)\leq v_{j,\varepsilon}(t,x)-\mathcal{M}_{j}^{k}v_{\varepsilon}(t,x)+C\varepsilon\leq -k+C\varepsilon.$$
Furthermore, since $|v^{\varepsilon}|_1\leq C$ for
$v^{\varepsilon}=(v_1^{\varepsilon},\dots,v_m^{\varepsilon})$, we
also have
\begin{align*}
v_{j}^{\varepsilon}(t-\tau,x-e)-\mathcal{M}_{j}^{k}v^{\varepsilon}(t-\tau,x-e)&\leq
v_{j}^{\varepsilon}(t,x)-\mathcal{M}_{j}^{k}v^{\varepsilon}(t,x)+C(|\tau|^{\frac12}+|e|)\\
&\leq -k+C\varepsilon+2C\varepsilon,
\end{align*}
for any $(\tau,e)\in \Theta^{\varepsilon}$. If we then choose $k=
4C\varepsilon$, we obtain that, for any $(\tau,e)\in
\Theta^{\varepsilon}$,
$$v_{j}^{\varepsilon}(t-\tau,x-e)-\mathcal{M}_{j}^{k}v^{\varepsilon}(t-\tau,x-e)<0.$$
Therefore, the point $(t-\tau,x-e)$, for $(\tau,e)\in
\Theta^{\varepsilon}$, is in the continuation region of
(\ref{Ptbedswitching1}). Thus,
$$-\partial_{t}v_{j}^{\varepsilon}(t-\tau,x-e)+\inf_{(\tau,e)\in\Theta^{\varepsilon}}\mathcal{L}^{q_{j}}\left(t,x,\partial_{x}v_{j}^{\varepsilon}(t-\tau,x-e),\partial_{xx}v_{j}^{\varepsilon}(t-\tau,x-e)\right)=0,$$
and, in turn,
$$-\partial_{t}v_{j}^{\varepsilon}(t-\tau,x-e)+\mathcal{L}^{q_{j}}\left(t,x,\partial_{x}v_{j}^{\varepsilon}(t-\tau,x-e),\partial_{xx}v_{j}^{\varepsilon}(t-\tau,x-e)\right)\ge0.$$
Using the definition of $v_{j,\varepsilon}$ and that
$\mathcal{L}^{q_{j}}$ is linear in $\partial_{x}v_{j}^{\varepsilon}$
and $\partial_{xx}v_{j}^{\varepsilon}$, we further have
\begin{align}\label{supsolution}
&-\partial_{t}v_{j,\varepsilon}(t,x)+\mathcal{L}^{q_{j}}\left(t,x,\partial_{x}v_{j,\varepsilon}(t,x),\partial_{xx}v_{j,\varepsilon}(t,x)\right)\\
=&\ \int_{-\varepsilon^2< \tau< 0}\int_{|e|< \varepsilon}
\left(-\partial_{t}v_{j}^{\varepsilon}(t-\tau,x-e)+\mathcal{L}^{q_{j}}\left(t,x,\partial_{x}v_{j}^{\varepsilon}(t-\tau,x-e),\partial_{xx}v_{j}^{\varepsilon}(t-\tau,x-e)\right)\right)\notag\\
&\times\rho_{\varepsilon}(\tau,e)ded\tau \ge 0\notag.
\end{align}
Next, we observe that, for $(t,x)\in\bar{Q}_{T-\Delta}$, the
definition of $j$ implies that
$w_{\varepsilon}(t,x)=v_{j,\varepsilon}(t,x)$ and
$w_{\varepsilon}(t+\Delta,\cdot)\leq
v_{j,\varepsilon}(t+\Delta,\cdot).$ Then, applying Proposition
\ref{scheme property} (ii) (iv) and estimate (\ref{convoproperty}),
we obtain that, for any $(t,x)\in\bar{Q}_{T-\Delta}$,
\begin{align*}
 &\ S(\Delta,t,x,w_{\varepsilon}(t,x),w_{\varepsilon}(t+\Delta,\cdot))\\
 \ge &\
 S(\Delta,t,x,v_{j,\varepsilon}(t,x),v_{j,\varepsilon}(t+\Delta,\cdot))\\
 \ge &\
 -\partial_{t}v_{j,\varepsilon}(t,x)+\sup_{q\in\mathbb{R}^{n}}\mathcal{L}^{q}(t,x,\partial_{x}v_{j,\varepsilon}(t,x),\partial_{xx}v_{j,\varepsilon}(t,x))-C\Delta\varepsilon^{-3}\\
 \ge &\
  -\partial_{t}v_{j,\varepsilon}(t,x)+\mathcal{L}^{q_{j}}(t,x,\partial_{x}v_{j,\varepsilon}(t,x),\partial_{xx}v_{j,\varepsilon}(t,x))-C\Delta\varepsilon^{-3}
  \ge
  -C\Delta\varepsilon^{-3},
\end{align*}
for some constant $C$ depending only on $M$ and $T$, where we used
(\ref{supsolution}) in the last inequality. In turn, the comparison
result in Proposition \ref{schemecomparison} implies that
$$u^{\Delta}-w_{\varepsilon}\leq \sup_{\bar{Q}_{T}\backslash \bar{Q}_{T-\Delta}}(u^{\Delta}-w_{\varepsilon})^{+}+C(T-t)\Delta\varepsilon^{-3}\ \ \text{in}\ \bar{Q}_T.$$
Combining the above inequality with (\ref{errtogether}), we further
get
\begin{align*}
u^{\Delta}-u^{m} = &\ (u^{\Delta}-w_{\varepsilon})+(w_{\varepsilon}-u^{m})\\
 \leq &\
 \sup_{\bar{Q}_{T}\backslash
 \bar{Q}_{T-\Delta}}(u^{\Delta}-w_{\varepsilon})^{+}+C(T-t)\Delta\varepsilon^{-3}+
 C(\varepsilon+k^{\frac13}+k^{\frac23})\\
 \leq &\
 \sup_{\bar{Q}_{T}\backslash\bar{Q}_{T-\Delta}}(u^{\Delta}-u^{m})^{+}+C(\varepsilon+\varepsilon^{\frac13}+\varepsilon^{\frac{2}{3}}+\Delta\varepsilon^{-3})\\
 \leq &\
 \sup_{\bar{Q}_{T}\backslash\bar{Q}_{T-\Delta}}(u^{\Delta}-u^{m})^{+}+C\Delta^{\frac{1}{10}}\
 \ \text{in}\ \bar{Q}_T,
\end{align*}
where we {used} $k=4C\varepsilon$ in the second to last inequality,
and {chose} $\varepsilon={\Delta}^{\frac{3}{10}}$ in the last
inequality.
\end{proof}\\

We are now ready to establish a lower bound for the approximation
{error}.
\begin{theorem}\label{theorem_error_2}
Suppose that Assumption \ref{data assumption} holds. Let
$u^{\Delta}\in\mathcal{C}_b(\bar{Q}_{T})$ satisfy the approximation
scheme (\ref{semischeme}) and $u\in\mathcal{C}_b^1(\bar{Q}_{T})$ be
the unique viscosity solution of equation (\ref{PDE_1}). Then, there
exists a constant $C$, depending only on $M$ and $T$, such that
$$u-u^{\Delta}\ge -C\Delta^{\frac{1}{10}}\ \ \text{in}\ \bar{Q}_T.$$
\end{theorem}

\begin{proof}
Proposition \ref{lowbd1} yields
\begin{align*}
u^{\Delta}-u = &\ (u^{\Delta}-u^{m})+(u^{m}-u)\\
 \leq &\
 \sup_{\bar{Q}_{T}\backslash\bar{Q}_{T-\Delta}}(u^{\Delta}-u^{m})^{+}+C\Delta^{\frac{1}{10}}+(u^m-u)\\
 \leq &\
 \sup_{\bar{Q}_{T}\backslash
 \bar{Q}_{T-\Delta}}(u^{\Delta}-u)^{+}+C\Delta^{\frac{1}{10}}+
 \sup_{\bar{Q}_{T}\backslash \bar{Q}_{T-\Delta}}(u-u^{m})^{+}+(u^{m}-u)\\
 \leq &\
 C\Delta^{\frac{1}{10}}+\sup_{\bar{Q}_{T}\backslash
 \bar{Q}_{T-\Delta}}(u-u^{m})^{+}+(u^{m}-u),
\end{align*}
where we used estimate (\ref{estimate_final_interval}) in the last
inequality. Sending $m\to\infty$ and using (\ref{lemma1}), we
conclude.
\end{proof}

%%%%%%%%%%%%%%%%%%%%%%%%%%%%%%%%%%%%%55
\section{A numerical example}
We present a numerical result, applying the approximation scheme
(\ref{semischeme}) for the case
$$\sigma(t,x)=1,\ b(t,x)=0,\ H(t,x,p)=p^2/2,\ T=1.$$
We also choose $U(x)=0\vee x\wedge K$ in the semilinear PDE
(\ref{PDE_1}). Then the equation in (\ref{PDE_1}) becomes the
Cole-Hopf equation (see \cite{Evans}):
\begin{equation}\label{Eg1}
-\partial_tu(t,x)-\frac{1}{2}\partial_{xx}u(t,x)+\frac12(\partial_xu(t,x))^2=0.
\end{equation}
% We solve on a space interval $\Omega=[0,10]$, with boundary conditions
%\begin{equation}\label{Eg1bc}
%u(t,0)=0; \ \ \ u(t,10)=K, \ \ \ t\in[0,T].
%\end{equation}
It is well known that, by the Cole-Hopf transformation (see
\cite{Evans} and \cite{Zari}), the function $v(t,x):=e^{-u(t,x)}$
satisfies the heat equation
$$\partial_tv(t,x)+\frac{1}{2}\partial_{xx}v(t,x)=0,$$
with $v(T,x)=e^{-U(x)}=e^{-0\vee x\wedge K}.$ In turn,
\begin{align*}
v(t,x)=&\ \Phi(-\frac{x}{\sqrt{T-t}})+e^{-x+(T-t)/2}\left(\Phi(\frac{K-x+T-t}{\sqrt{T-t}})-\Phi(\frac{-x+T-t}{\sqrt{T-t}})\right)\\
&\ +e^{-K}\Phi(-\frac{K-x}{\sqrt{T-t}}),
\end{align*}
where $\Phi$ is the standard normal cumulative distribution function
and, thus, we obtain the explicit solution $u(t,x)=-\log v(t,x).$

We use this exact solution as a benchmark, and compare it with the
approximate solution obtained by the approximation scheme
(\ref{semischeme}). Moreover, we also compare our results with the
ones obtained via the standard Howard's finite difference (FD)
algorithm (see, for example, \cite{BMZ} for a detailed discussion of
Howard's FD scheme).

Since Howard's scheme is based on the finite difference method, for
the comparison purpose, we also numerically compute the conditional
expectation appearing in the backward operator
$\mathbf{S}_t(\Delta)$ (cf. (\ref{semigroupequation1})) using the
finite difference method. However, we emphasize that, different from
Howard's scheme, the splitting approximation itself does not depend
on the finite difference method {(as long as one can find an
efficient way to compute conditional expectations, e.g. the multi-level Monte Carlo approach \cite{MLMC_method}, the least squares Monte Carlo approach \cite{LS_method}, the cubature approach \cite{cubature_method}, and etc).} Hence, our
approximation scheme can be potentially used to numerically solve
high dimensional PDEs without the ``curse of dimensionality" issue.

To numerically compute the finite-dimensional minimization problem
in the backward operator $\mathbf{S}_t(\Delta)$ (cf.
(\ref{semigroupequation1})), since the finite difference method already provides us with all the points to be compared, we use the simple brute force method to find the minimizers and minimal values\footnote{In general, we may implement the standard Nelder-Mead simplex
algorithm (see \cite{NM}), which is commonly used in the literature when the derivatives of the objective function in the minimization problem are not known.}.

Figures 1 and 2 demonstrate the performance of the approximation
scheme (\ref{semischeme}) with the parameter $K=5$. They illustrate
how the approximate solutions converge as we increase the number of
time steps $T/\Delta$. For our parameter values, $\Delta=0.1$ (so
$T/\Delta=10$) is sufficient for the approximate solutions to
converge, as the relative error is already negligible ($0.056\%$).

{Figure 3 compares the values numerically computed by the
approximation scheme (\ref{semischeme}) and the Howard's FD scheme
with different time steps. It shows that the approximation scheme
gives a better approximation than the Howard's scheme does. In
particular, we observe that when the time step $\Delta=0.1$ (so $T/\Delta=10$),
the numerical solution computed by our approximation scheme is far more accurate than
the one computed by the FD scheme. The relative error is $0.056\%$ for the former and $0.142\%$ for the latter.
It also shows
that the approximation scheme converges linearly with time step
$\Delta$, and this is consistent with our theoretical results in
Theorem \ref{smoothcase}. Table 1 further compares the computation errors and costs between the
approximation scheme (\ref{semischeme}) and the Howard's FD scheme. Since there involves an additional minimization step
in the approximation scheme (\ref{semischeme}), its computation costs are higher than the FD scheme. However, we observe that when the time step is small (e.g. $\Delta=0.1$),
the computation times for both schemes are extremely fast (less than $0.05$ second). }

%I changed the function 'fminsearch' to 'fminbnd' which is more
%efficient in our context. This result is in the tab 'New' in the
%spreadsheet. You can see the speed doubled immediately. Then I tried
%to compute the min by myself so that we can probably avoid using
%'fmin' functions at all and increase the speed further more. The
%idea is that instead of using 'fmin' functions directly, we can
%further discretise the x-axis intervals into r subintervals and
%compute the values on all of those 'subpoints' and take min through
%them all. Obviously, the greater r is, the more precise the minimum
%can be, and the longer time it will take. So I tried r = 10, 15 and
%20 and the results are below the yellow tabs in the spreadsheet.
%Surprisingly, the speed is way more faster without compromising the
%precision too much.

%%%%%%%%%%%%%%%%%%%%%%%%%%%%%%%%%%%%%%%%%%%%%%%%%%%%%%%%%%%%%%%%5
\section{Conclusions}

We proposed an approximation scheme for a class of semilinear
parabolic equations whose Hamiltonian is convex and coercive to the
gradients. The scheme is based on splitting the equation in two
parts, the first corresponding to a linear parabolic equation and
the second to a Hamilton-Jacobi equation. The solutions of these
equations are approximated using, respectively, the Feynman-Kac and
the Hopf-Lax formulae. We established the convergence of the
approximation scheme and determined the convergence rate, combining
Krylov's shaking coefficients technique and Barles-Jakobsen's
optimal switching approximation. One of the key steps is the
derivation of a consistency error via convex duality arguments,
using the convexity of the Hamiltonian in an essential way.

The approach and the results herein may be extended in various
directions. Firstly, one may consider problem (\ref{PDE_1}) in a
bounded domain, an undoubtedly important case since various
applications are cast in such domains (e.g. utilities defined in
half-space, constrained risk measures, etc.) However, various
non-trivial technical difficulties arise. Some recent works on such
problems using other approaches can be found in \cite{CS},
\cite{Krylov2} and \cite{Reisinger}.

Secondly, one may consider variational versions of problem
(\ref{PDE_1}). These are naturally related to optimal stopping and
to singular stochastic optimization problems, both directly related
to various applications with early-exercise, fixed and/or
proportional transaction costs, irreversible investment decisions,
etc. Recent results in this direction that use some of the ideas
developed herein can be found in \cite{Huang}.

%allow for constraints on both the solution and its gradients, which
%correspond to, respectively, optimal stopping and singular control
%optimization. See \cite{Huang} for such an extension. Another
%important extension is to allow for bounded domains. This is
%undoubtedly a very important extension, for in many applications the
%states of the equation are constrained in a bounded domain. However,
%it is far more challenging to prove the convergence rate of the
%approximate solutions in bounded domains (see \cite{CS},
%\cite{Krylov2} and \cite{Reisinger} for some recent developments in
%this direction). We leave such an extension for our future research.

%\section{Weak approximation of the degenerate SDE}
%
%To be done!

%%%%%%%%%%%%%%%%%%%%%%%%%%%%%%%%%%%%%%%%%%%%%%%%%%%%%%%%%%%%%%%55
\begin{appendix}

\section{Proofs of Propositions \ref{solutionproperty} and \ref{upperbound_1}}\label{App1}

We note that equation (\ref{PDE_1}) is a special case (choosing
$\varepsilon=0$) of the HJB equation (\ref{PtbedPDE}). Therefore, we
omit the proof of Proposition \ref{solutionproperty} and only prove
Proposition \ref{upperbound_1}.

We first show that there exists a bounded solution to
(\ref{PtbedPDE}). To this end, using the convex dual function
$L^{\theta}(t,x,q):=\sup_{p\in\mathbb{R}^n}(p\cdot
q-H^{\theta}(t,x,p))$, we rewrite (\ref{PtbedPDE}) as
\begin{equation}\label{HJB_equation}
\left\{\begin{array}{ll}
\displaystyle-\partial_{t}u^{\varepsilon}+{\sup_{\theta\in\Theta^{\varepsilon},q\in\mathbb{R}^{n}}}\mathcal{L}^{\theta,q}\left(t,x,\partial_{x}u^{\varepsilon},\partial_{xx}u^{\varepsilon}\right)=0&\text{in}\  Q_{T+\varepsilon^2};\\
\displaystyle u^{\varepsilon}(T+\varepsilon^2,x)=U(x)&\text{in}\
\mathbb{R}^n,
\end{array}\right.
\end{equation}
where
\begin{equation*}
\mathcal{L}^{\theta,q}(t,x,p,X)=-\frac{1}{2}\text{Trace}\left(\sigma^{\theta}\sigma^{\theta^T}(t,x)X\right)-(b^{\theta}(t,x)-q)\cdot
p-L^{\theta}(t,x,q).
\end{equation*}
We also introduce the stochastic control problem
$$u^{\varepsilon}(t,x)={\inf_{\theta\in\Theta^{\varepsilon}[t,T+\varepsilon^2],q\in\mathbb{H}^{2}[t,T+\varepsilon^{2}]}}\mathbf{E}\left[\int_{t}^{T+\varepsilon^{2}}L^{{\theta_{s}}}\left(s,X_{s}^{t,x;\theta,q},q_{s}\right)ds+U\left(X_{T+\varepsilon^{2}}^{t,x;\theta,q}\right)|\mathcal{F}_t\right],$$
with the controlled state equation
$$dX_{s}^{t,x;\theta,q}=\left(b^{{\theta_{s}}}(s,X_{s}^{t,x;\theta,q})-q_{s}\right){ds}+\sigma^{{\theta_{s}}}\left(s,X_{s}^{t,x;\theta,q}\right)dW_{s},$$
where $\Theta^{\varepsilon}[t,T+\varepsilon^2]$ is the space of
$\Theta^{\varepsilon}$-valued progressively measurable processes
$(\tau_s,e_s)$ and $\mathbb{H}^{2}[t,T+\varepsilon^2]$ is the space
of square-integrable progressively measurable processes $q_s$, for
$s\in[t,T+\varepsilon^2]$. Next, we identify its value function with
a bounded viscosity solution to (\ref{HJB_equation}). For this, we
only need to establish upper and lower bounds for the value function
$u^{\varepsilon}(t,x)$ and, in turn, use standard arguments as in
\cite{Pham} and \cite{Touzi}.

To find an upper {bound} for $u^{\varepsilon}$,  we choose an
arbitrary perturbation parameter process
$\theta\in\Theta^{\varepsilon}[t,T+\varepsilon]$ and choose
$\hat{q}$ with $\hat{q}_{s}\equiv 0$. Then, Proposition
\ref{transformlemma} (ii) yields
\begin{align*}
u^{\varepsilon}(t,x)\le &\ \mathbf{E}\left[\int_{t}^{T+\varepsilon^{2}}L^{{\theta_{s}}}(s,X_{s}^{t,x;\theta,\hat{q}},0)ds+U(X_{T+\varepsilon^{2}}^{t,x;\theta,\hat{q}})|\mathcal{F}_t\right] \\
 \le &\
 (T+\varepsilon^{2}-t)|L^{*}(0)|+M \le (T+1)|L^{*}(0)|+M.
\end{align*}

For the lower bound, we use again Proposition \ref{transformlemma}
(ii) to obtain that $L_{*}(q)\ge -H^{*}(0) \ge -|H^{*}(0)|$, for any
$q\in\mathbb{R}^{n}$. In turn, for any
$(\theta,q)\in\Theta^{\varepsilon}[t,T+\varepsilon^2]\times\mathbb{H}^2[t,T+\varepsilon^2]$,
\begin{align*}
 &\ \mathbf{E}\left[\int_{t}^{T+\varepsilon^{2}}L^{{\theta_{s}}}(s,X_{s}^{t,x;\theta,q},q_{s})ds+U(X_{T+\varepsilon^{2}}^{t,x;\theta,q})|\mathcal{F}_t\right] \\
 \ge &\
 \mathbf{E}\left[\int_{t}^{T+\varepsilon^{2}}L_{*}(q_{s})ds+U(X_{T+\varepsilon^{2}}^{t,x;\theta,q})|\mathcal{F}_t\right] \\
 \ge &\
 -(T+\varepsilon^{2}-t)|H^{*}(0)|-M \ge -(T+1)|H^{*}(0)|-M,
\end{align*}
and, thus, $u^{\varepsilon}(t,x)\ge -(T+1)|H^{*}(0)|-M$ and
$|u^{\varepsilon}|_{0}\le C$, for some constant $C$ independent of
$\varepsilon$.

The uniqueness of the viscosity solution is a direct consequence of
the continuous dependence result, presented next. Its proof follows
along similar arguments as in Theorem A.1 of \cite{Jakobsen} and is
thus omitted.

\begin{lemma}\label{Ptbedproperty}
For any $s\in(0,{T+\varepsilon^{2}}]$, let $u\in USC(\bar{Q}_{s})$
be a bounded from above viscosity subsolution of (\ref{PtbedPDE})
with coefficients $\sigma^\theta,b^\theta$ and $H^{\theta}$, and
$\bar{u}\in LSC(\bar{Q}_{s})$ be a bounded from below viscosity
supersolution of {(\ref{PtbedPDE})} with coefficients
$\bar{\sigma}^\theta,\bar{b}^\theta$ and $\bar{H}^{\theta}$. Suppose
that Assumption \ref{data assumption} holds for both sets of
coefficients with respective constants $M$ and $\bar{M}$, uniformly
in $\theta\in\Theta^{\varepsilon}$, and that either
$u(s,\cdot)\in\mathcal{C}_b^1(\mathbb{R}^n)$ or
$\bar{u}(s,\cdot)\in\mathcal{C}_b^1(\mathbb{R}^n)$. Then, there
exists a constant $C$, depending only on $M$, $\bar{M}$,
$[u(s,\cdot)]_1$ or $[\bar{u}(s,\cdot)]_1$, and $s$, such that, for
$(t,x)\in\bar{Q}_s$,
\begin{equation}\label{ctdependence}
u-\bar{u}\le
C\left(|(u(s,\cdot)-\bar{u}(s,\cdot))^{+}|_{0}+\sup_{\theta\in\Theta^{\varepsilon}}\left\{|\sigma^\theta-\bar{\sigma}^\theta|_{0}+|b^\theta-\bar{b}^\theta|_{0}\right\}+\sup_{\theta\in\Theta^{\varepsilon}}|H^\theta-\bar{H}^\theta|_{0}\right).
\end{equation}
\end{lemma}

The $x$-regularity of $u^{\varepsilon}$ follows easily from
(\ref{ctdependence}) by choosing $u=u^{\varepsilon}$,
$\bar{u}=u^{\varepsilon}(\cdot,\cdot+e)$ {and
$s=T+\varepsilon^{2}$}.

To get the time regularity, we work as follows. Firstly, let
$\rho(x)$ be a $\mathbb{R}_+$-valued smooth function with compact
support $B(0,1)$ and mass $1$, and introduce the sequence of
mollifiers
$\rho_{\varepsilon}(x):=\frac{1}{\varepsilon^{n}}\rho\left(\frac{x}{\varepsilon}\right).$
For $0\leq t < s \leq T+\varepsilon^{2}$, let $u_{\varepsilon'}$ be
the unique bounded solution of (\ref{PtbedPDE}) in $Q_{s}$ with
terminal condition
$u_{\varepsilon'}(s,x)=u^{\varepsilon}(s,\cdot)*\rho_{\varepsilon'}(x)$,
for some $\varepsilon'>0$. It then follows from (\ref{ctdependence})
that, for $(t,x)\in\bar{Q}_{s}$,
$$u^{\varepsilon}-u_{\varepsilon'}\leq C|(u^{\varepsilon}(s,\cdot)-u_{\varepsilon'}(s,\cdot))^{+}|_{0}\leq C[u^{\varepsilon}(s,\cdot)]_{1}\varepsilon'\le C\varepsilon'.$$
Similarly, we also have $u_{\varepsilon'}-u^{\varepsilon}\leq
C\varepsilon'$.

On the other hand, standard properties of mollifiers imply that
$|\partial_x^{j}u_{\varepsilon'}(s,\cdot)|_0\leq
C\varepsilon'^{1-j}$. Next, define the functions
$w^{+}_{\varepsilon'}(t,x):=u_{\varepsilon'}(s,x)+(s-t)C_{\varepsilon'}$
and
$w^{-}_{\varepsilon'}(t,x):=u_{\varepsilon'}(s,x)-(s-t)C_{\varepsilon'}$,
where $C_{\varepsilon'}=(\frac{1}{\varepsilon'}+1)C$, for some
constant $C$ independent of $\varepsilon$. We easily deduce that
they are, respectively, bounded supersolution  and subsolution of
(\ref{PtbedPDE}) in $Q_{s}$, with the same terminal condition
$w^{+}_{\varepsilon'}(s,x)=w^{-}_{\varepsilon'}(s,x)=u_{\varepsilon'}(s,x)$.
Thus, by (\ref{ctdependence}), we have
$w^{-}_{\varepsilon'}(t,x)\leq u_{\varepsilon'}(t,x) \leq
w^{+}_{\varepsilon'}(t,x),$ for $(t,x)\in\bar{Q}_{s}$, which in turn
implies that $|u_{\varepsilon'}(t,x)-u_{\varepsilon'}(s,x)|\leq
C_{\varepsilon'}|s-t|.$ Choosing $\varepsilon'=\sqrt{|s-t|}$, we
then obtain that
\begin{align*}
|u^{\varepsilon}(t,x)-u^{\varepsilon}(s,x)|\leq &\ |u^{\varepsilon}(t,x)-u_{\varepsilon'}(t,x)|+|u_{\varepsilon'}(t,x)-u_{\varepsilon'}(s,x)|+|u_{\varepsilon'}(s,x)-u^{\varepsilon}(s,x)|\\
 \leq &\
 2C\varepsilon'+C_{\varepsilon'}|s-t|
 \leq
 C(\varepsilon'+\frac{|s-t|}{\varepsilon'}+|s-t|)\leq
 C\sqrt{|s-t|},
\end{align*}
which, together with {the boundedness and} the $x$-regularity of
$u^{\varepsilon}$, implies that $|u^{\varepsilon}|_{1}\le C$.

Finally, note that $u(t,x)$ is also the bounded viscosity solution
of (\ref{PtbedPDE}) when $\sigma^\theta\equiv\sigma$,
$b^\theta\equiv b$ and $H^{\theta}\equiv H$. Applying
(\ref{ctdependence}) once more and using the regularity of $\sigma$,
$b$, $H$ and $u^\varepsilon$, we deduce that
\begin{align*}
u^\varepsilon-u \le &\ C\left(|(u^\varepsilon(T,\cdot)-u(T,\cdot))^{+}|_0+\sup_{\theta\in\Theta^{\varepsilon}}\left\{|\sigma^\theta-\sigma|_0+|b^\theta-b|_0\right\}+\sup_{\theta\in\Theta^{\varepsilon}}|H^{\theta}-H|_{0}\right)\\
 \le &\
  C\left(|u^\varepsilon(T,\cdot)-u^\varepsilon(T+\varepsilon^2,\cdot)|_0+\varepsilon\right)\le
  C\varepsilon\ \ \text{in}\ \bar{Q}_{T}.
\end{align*}
 Similarly, we also have $u-u^\varepsilon\leq
C\varepsilon$, and we easily conclude.

\end{appendix}

%%%%%%%%%%%%%%%%%%%%%%%%%%%%%%%%%%%%%%%%%%%%%%%%%%%55

%%%%%%%%%%%%%%%%%%%%%%%%%%%%%%%%%%%%%%%%%%%%%%%%%%%%%%%%%%%
\newpage
\begin{figure}[h!]
\centering
\includegraphics[width=0.8\textwidth]{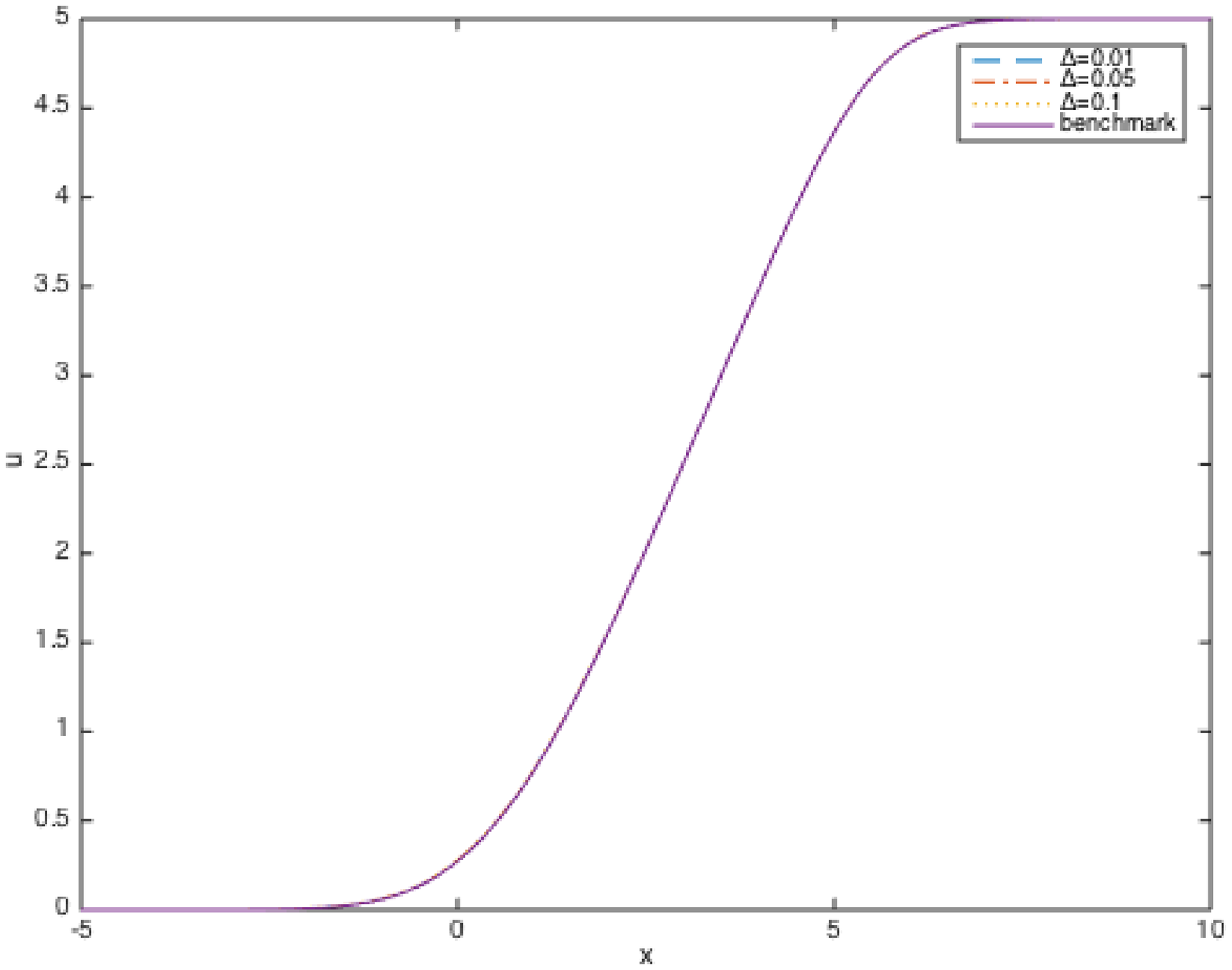}
\caption{Approximate values of $u(0,x)$ with various time steps
$\Delta=0.01/0.05/0.1$.}
%\end{figure}
%\begin{figure}[h]
%\centering
\includegraphics[width=0.8\textwidth]{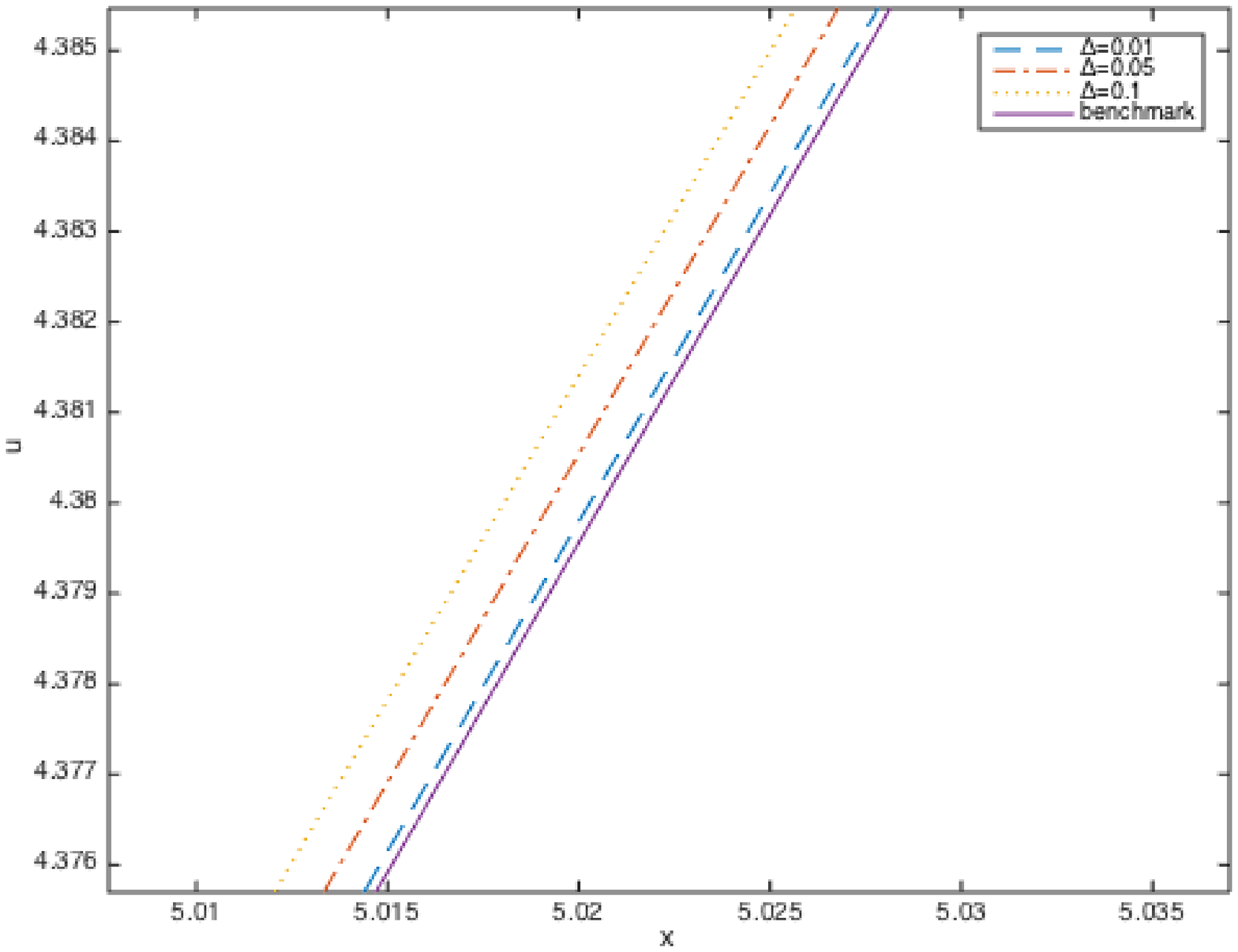}
\caption{Approximate values of $u(0,x)$ with various time steps
$\Delta=0.01/0.05/0.1$. The figure zooms in Fig. 1.1.}
\end{figure}

\newpage
\begin{figure}[h!]
\centering
\includegraphics[width=0.8\textwidth]{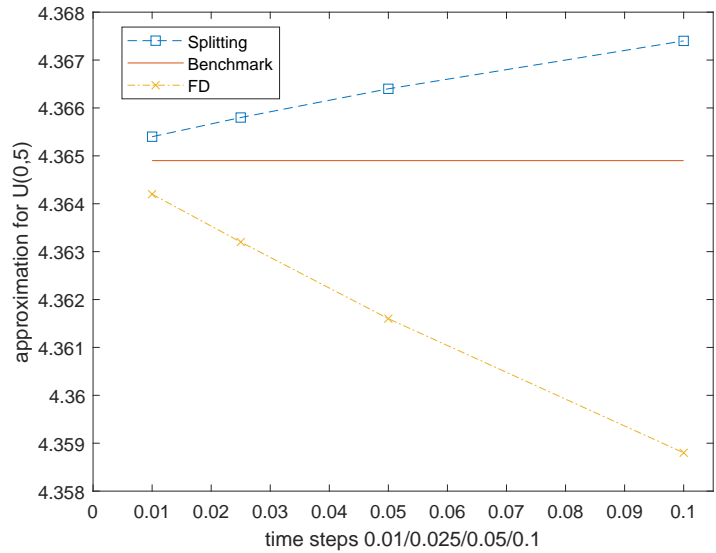}
\caption{Comparison of exact value and approximate values for
$u(0,5)$ via the approximation scheme (\ref{semischeme}) and the
Howard's FD scheme with various time steps
$\Delta=0.01/0.025/0.05/0.1$.}
\end{figure}

\begin{table}[h]
\centering
 \begin{tabular}{|l|*{5}{c|}}
 \hline  \backslashbox{numerical schemes}{time steps} & $0.01$ & $0.25$ & $0.05$& $0.1$\\
 \hline splitting approx. value & 4.3655 & 4.3658 & 4.3664 & 4.3674\\
 \hline \ \ \ \ \ \ \ \ \ \ \ \ approx. error & 0.012\% & 0.02\% & 0.032\% & 0.056\%  \\
 \hline \ \ \ \ \ \ \ \ \ \ \ \ running time (in seconds) & 18.78 & 1.07 & 0.16 & 0.04  \\
 \hline FD approx. value & 4.3642 & 4.3632 & 4.3616 & 4.3588 \\
 \hline \ \ \ \ \ \ approx. error & 0.016\% & 0.039\% & 0.076\% & 0.142\%  \\
 \hline \ \ \ \ \ \ running time (in seconds) & 7.01 & 0.43 & 0.03 & 0.01 \\
 \hline
 \end{tabular}
\vspace{0.3cm}
\caption{Comparison of running errors and costs for approximating $u(0,5)$ via the approximation scheme (\ref{semischeme}) and the
Howard's FD scheme with various time steps
$\Delta=0.01/0.025/0.05/0.1$.}
\end{table}


\small
\begin{thebibliography}{1}

\bibitem{BJ0}
Barles, G. and E. R. Jakobsen. {Error bounds for monotone
approximation schemes for {H}amilton-{J}acobi-{B}ellman equations}.
{\it SIAM Journal on Numerical Analysis}, 43(2): 540-558, 2005.

\bibitem{BJ}
Barles, G. and E. R. Jakobsen. {Error bounds for monotone
approximation schemes for parabolic {H}amilton-{J}acobi-{B}ellman
equations}. {\it Mathematics of Computation}, 76: 1861-1893, 2007.

\bibitem{BJ1}
Barles, G. and E. R. Jakobsen. On the convergence rate of
approximation schemes for Hamilton-Jacobi-Bellman equations. {\it
M2AN Math. Model. Numer. Anal.}, 36(1): 33-54, 2002.

\bibitem{Barles}
Barles, G. and P. E. Souganidis. Convergence of approximation
schemes for fully nonlinear second order equations. {\it Asymptotic
Analysis}, {4(3)}: 271--283, 1991.

\bibitem{Erhan} Bayraktar, E. and A. Fahim.
A stochastic approximation for fully nonlinear free boundary
problems. {\em Numerical Methods for Partial Differential
Equations}, 30(3): 902-929, 2014.

\bibitem{BMZ} Bokanowski, O., S. Maroso, and H. Zidani.
Some convergence results for Howard's algorithm. {\it SIAM Journal
on Numerical Analysis}, 47(4): 3001-3026, 2009.

\bibitem{BPZ} Bokanowski, O., A. Picarelli, and H. Zidani. Dynamic programming
and error estimates for stochastic control problems with maximum
cost. {\it Applied Mathematics \& Optimization}, 71(1): 125-163,
2015.

%\bibitem{MR2547456} Carmona, R. (editor),
%Indifference pricing, theory and applications, {\it Princeton
%University Press}, (2009).

\bibitem{Touzi2}
Bouchard, B. and N. Touzi.
\newblock{Discrete-time approximation and {M}onte-{C}arlo simulation of backward stochastic differential
equations}.
\newblock{\em Stochastic Processes and their Applications},
111(2), 175--206, 2004.

\bibitem{CS} Caffarelli, L.A. and P. E. Souganidis. A rate of convergence for
monotone finite difference approximations to fully nonlinear
uniformly elliptic PDEs. {\it Communications on Pure and Applied
Mathematics}, 61(1):1-7, 2008.

\bibitem{CR} Chassagneux, J.F. and A. Richou. Numerical
simulation of quadratic BSDEs. {\it The Annals of Applied
Probability}, 26(1): 262--304, 2016.

\bibitem{CHLZ}
Chong, W.F., Y. Hu, G. Liang and T. Zariphopoulou. An ergodic BSDE
approach to forward entropic risk measures: representation and
large-maturity behavior. \emph{Finance and Stochastics}, 23(1):
239--273, 2019.

%\bibitem{CIL} Crandall, M. G., H. Ishii, and P.-L. Lions.
%User's guide to viscosity solutions of second order partial
%differential equations. {\it Bull. Amer. Math. Soc.}, 27(1):1-67,
%1992.

%\bibitem{MR2053051}
%Delarue, F.,
%\newblock{Estimates of the solutions of a system of quasi-linear {PDE}s: a probabilistic scheme},
%\newblock{\em S\'eminaire de {P}robabilit\'es {XXXVII}, Lecture Notes in Math., Spring},
%1832, (2003), 290--332.

\bibitem{DHB} Delbaen, F., Y. Hu, and X. Bao. Backward SDEs with
superquadratic growth. {\em Probability Theory and Related Fields},
150(1), 145-192, 2011.



\bibitem {Peng}
El Karoui, N. and R. Rouge. Pricing via utility maximization and
entropy. {\it Mathematical Finance} {10}, 259--276. 2000.

%El Karoui, N., S. Peng, and M. C. Quenez.
%\newblock Backward stochastic differential equations in finance,
%\newblock {\em Mathematical Finance},
%7(1), 1--71, 1997.

\bibitem{Evans}
Evans, L. C. {\it Partial Differential Equations}. {Graduate Studies
in Mathematics, Americal Mathematical Society}, 1998.

\bibitem{EF}
Evans, L. C. and A. Friedman. Optimal stochastic switching and the
Dirichlet problem for the Bellman equation. {\it Trans. Amer. Math.
Soc}, 253:365-389, 1979.

\bibitem{FTW} Fahim A., N. Touzi, and X. Warin. A probabilistic numerical method
for fully nonlinear parabolic PDEs. {\it The Annals of Applied
Probability}, 21: 1322--1364, 2011.

\bibitem{MLMC_method} Giles, M. B. Multilevel monte carlo path simulation. {\it Operations Research}, 56(3), 607--617, 2008.

%\bibitem{Fleming} Fleming, W. H. and H. M. Soner. {\it Controlled Markov Processes and
%Viscosity Solutions}, {Springer}, 2006.

%\bibitem{Henderson1} Henderson, V., Valuation of claims on nontraded assets using
%utility maximization, {\it Mathematical Finance}, {12}, (2002),
%351--373.
%
\bibitem{HL} Henderson, V. and G. Liang. Pseudo linear pricing rule for utility indifference
valuation. {\it Finance and Stochastics}, {18(3)}:593--615, 2014.
%
%
\bibitem{HL1} Henderson, V. and G. Liang. A multidimensional exponential utility
indifference pricing model with applications to counterparty risk.
{\it  SIAM Journal on Control and Optimization}, 54(3): 690--717,
2016.
%
\bibitem{Hu} Hu, Y., P. Imkeller and M. M\"uller. Utility maximization
in incomplete markets. {\it The Annals of Applied Probability},
{15}: 1691--1712, 2005.

\bibitem{Huang} Huang, S. An approximation scheme for variational
inequalities with convex and coercive Hamiltonians. {\it Working
paper}, 2018, arXiv:1810.08842.

\bibitem{Jakobsen} Jakobsen, E. R.
On the rate of convergence of approximation schemes for Bellman
equations associated with optimal stopping time problems. {\it Math.
Models Methods Appl. Sci.}, 13(5): 613-644, 2003.

\bibitem{Kobylanski} Kobylanski, M. Backward stochastic differential
equations and partial differential equations with quadratic growth.
{\it The Annals of Probability}, {28}: 558--602, 2000.

\bibitem{Krylov1} Krylov, N. V.
On the rate of convergence of finite-difference approximations for
Bellman's equation. {\it St. Petersburg Math. J.}, 9(3):639-650,
1997.

\bibitem{Krylov} Krylov, N. V.
{On the rate of convergence of finite-difference approximations
              for {B}ellman's equations with variable coefficients}.
{\it Probability Theory and Related Fields}, {117(1)}: 1-16, 2000.

\bibitem{Krylov2} Krylov, N. V. On the rate of convergence of finite-difference
approximations for elliptic Isaacs equations in smooth domains. {\it
Communications in Partial Differential Equations}, 40(8): 1393-1407,
2015.

\bibitem{Lions} Lions, P. L. and B. Mercier. Splitting algorithms for the
sum of two nonlinear operators. {\it SIAM Journal on Numerical
Analysis}, 16(6): 964-979, 1979.

\bibitem{LS_method} Longstaff, F. A. and E. S. Schwartz. Valuing American options by simulation: a simple least-squares approach. {\it The Review of Financial Studies}, 14(1), 113-147, 2001.

\bibitem{cubature_method} Lyons, T. and N. Victoir. Cubature on Wiener space. {\it Proceedings of the Royal Society of London. Series A: Mathematical, Physical and Engineering Sciences}, 460(2041), 169--198, 2004.

\bibitem{Marchuk} Marchuk, G. I. Some application of splitting-up methods to the
solution of mathematical physics problems. {\it Apl. Mat.}, 13:
103-132, 1968.

\bibitem{NZ} Nadtochiy, S. and T. Zariphopoulou.
An approximation scheme for solution to the optimal investment
problem in incomplete markets. {\em SIAM J. Finan. Math.}, 4(1):
494-538, 2013.

\bibitem{NM}
Nelder, J. A. and R. Mead. A simplex method for function
minimization. {\em Computer Journal}, 7: 308-313, 1965.

\bibitem{Pham} Pham, H. {\it
Continuous-time Stochastic Control and Optimization with Financial
Applications}. {Springer}, 2009.

\bibitem{Reisinger} Picarelli, A., C. Reisinger, and J. Rotaetxe Arto.
{Error bounds for monotone schemes for parabolic
Hamilton-Jacobi-Bellman equations in bounded domains}. {\em Working
paper}, 2017, arXiv:1710.11284.

\bibitem{Tan} Tan X. A splitting method for fully nonlinear degenerate parabolic
PDEs. {\it Electron. J. Probab.}, 18(15): 1-24, 2013.

\bibitem{Tourin} Tourin, A. Splitting methods for Hamilton-Jacobi equations,
{\it Numerical Methods Partial Differential Equations}, 22: 381-396,
2006.

\bibitem{Touzi} Touzi, N. {\it Optimal Stochastic Control, Stochastic Target Problems, and Backward SDE}.
{Springer}, 2012.

\bibitem{Zari} Zariphopoulou, T.,
A solution approach to valuation with unhedgeable risks, {\it
Finance and Stochastics}, {5(1)}: 61--82, 2001.

\end{thebibliography}
\end{document}